\documentclass[10pt]{article}
\usepackage{lmodern}
\usepackage[english]{babel}
\usepackage[T1]{fontenc}
\usepackage{babelbib}
\usepackage{cancel}
\usepackage{amsmath}
\usepackage{amssymb}
\usepackage{indentfirst}
\usepackage[latin1]{inputenc}
\usepackage{geometry}
\usepackage{amsfonts}
\usepackage{graphicx}
\usepackage{float}
\usepackage{color}
\usepackage{hyperref}
\hypersetup{colorlinks,linktocpage}
\usepackage{verbatim}
\usepackage{enumerate}
 \usepackage[normalem]{ulem}

\newcommand{\M}{\mathcal{M}}
\newcommand{\V}{\mathcal{V}}
\renewcommand{\S}{\mathcal{S}}

\newcommand{\Z}{{\mathbb Z}}

\newcommand{\Q}{{\mathbb Q}}
\newcommand{\R}{{\mathbb R}}
\newcommand{\C}{{\mathbb C}}

\newcommand{\F}{{\mathcal F}}

\newcommand{\X}{{\mathbb X}}

\newcommand{\Hy}{{\mathbb H}}
\newcommand{\HTwo}{\Hy^2\times \Hy^2}

\newcommand{\h}{{\Gamma}}   

\newcommand{\bc}{\begin{center}}
\newcommand{\ec}{\end{center}}

\newcommand{\U}{{\cal U}}

\newcommand{\GL}{{\rm GL}}
\newcommand{\SL}{{\rm SL}}

\newcommand{\PSL}{{\rm PSL}}

\newcommand{\GEN}[1]{\langle #1 \rangle}

\renewcommand{\O}{\mathcal{O}}
\renewcommand{\P}{\mathbb{P}}

\newcommand{\te}{{\mathcal{T}}}

\newcommand{\inv}{^{-1}}

\newcommand{\matriz}[1]{\begin{array} #1 \end{array}}
\newcommand{\pmatriz}[1]{\left(\begin{array} #1 \end{array}\right)}

\newcommand{\qed}{\enspace\vrule  height6pt  width4pt  depth2pt}
\newenvironment{proof}{\par\noindent{\bf Proof.}}{$\qed$\par\bigskip}

\newtheorem{theorem}{Theorem}[section]
\newtheorem{definition}[theorem]{Definition}
\newtheorem{lemma}[theorem]{Lemma}
\newtheorem{corollary}[theorem]{Corollary}
\newtheorem{proposition}[theorem]{Proposition}
\newtheorem{remark}[theorem]{Remark}

\newtheorem{notation}[theorem]{Notation}

\title{Presentations of Groups Acting Discontinuously on Direct Products of Hyperbolic Spaces\thanks{ Mathematics subject Classification Primary [20G20, 22E40,
16S34, 16U60].
Keywords and phrases: Hilbert Modular Group, Discontinuous Action on Direct Product of Hyperbolic $2$-space,  Presentation, Fundamental Domain, Group Rings, Unit Group.
\newline The first author is supported in part by Onderzoeksraad of
Vrije Universiteit Brussel and Fonds voor
Wetenschappelijk Onderzoek (Flanders). The second author is supported by Fonds voor
Wetenschappelijk Onderzoek (Flanders)-Belgium.  The third author is partially supported by the Spanish Government under Grant MTM2012-35240 with "Fondos FEDER". }}

\author{E. Jespers \and A. Kiefer \and \'{A}. del R\'{\i}o }

\begin{document}

\maketitle

\begin{abstract}
The problem of describing the group of units $\U(\Z G)$ of the integral group ring $\Z G$  of a finite group $G$ has attracted a lot of attention and providing presentations for such groups is a fundamental problem. 
Within the context of orders, a central problem is to describe a presentation of the unit group of an order $\O$ in the simple epimorphic images $A$ of the rational group algebra $\Q G$.
Making use of the presentation part of Poincar\'e's Polyhedron  Theorem, Pita, del R\'{\i}o and Ruiz proposed such a method for a large family of finite groups $G$ and consequently Jespers, Pita, del R\'{\i}o, Ruiz and Zalesskii  described the structure  of $\U(\Z G)$ for a large family of finite groups $G$.
In order to handle many more groups, one  would like to extend Poincar\'e's Method to  discontinuous subgroups of the group of  isometries of a  direct product of hyperbolic spaces.
If the algebra $A$ has degree 2 then via the Galois embeddings of the centre of the algebra  $A$ one considers the group of reduced norm one elements of the order $\O$ as such a group and thus one would obtain a solution to the mentioned problem.
This would provide presentations of the unit group of orders in the simple components of degree 2 of $\Q G$ and in particular describe the unit group of $\Z G$ for every group $G$ with irreducible character degrees less than or equal to 2.
The aim of this paper is to initiate this approach by executing this method   on the Hilbert modular group, i.e. the projective linear group of degree two over the ring of integers in a real quadratic extension of the rationals. This group acts discontinuously on a direct product of two hyperbolic spaces of dimension two.
The fundamental domain constructed is an analogue of the Ford domain of a Fuchsian or a Kleinian group.
\end{abstract}

\section{Introduction}

The aim of this work is to generalize the presentation part of Poincar\'e's Polyhedron Theorem to a discontinuous group acting on a direct product of two copies of hyperbolic $2$-space.
Our motivation  comes from the investigations on the unit group $\U (\Z G)$ of the integral group ring $\Z G$.
One of the important  problems is to determine a presentation of this group (see \cite{SehgalBook}).
In order to make further progress, there  is a need  for finding new methods to determine  generators and next to deduce  a presentation.
In \cite{PRR}, Pita, del R\'{\i}o and Ruiz initiated in these investigations the use of   actions on hyperbolic spaces. This allowed them to obtain presentations for subgroups of finite index in $\U (\Z G)$ for some new class of finite groups $G$, called groups  of Kleinian type. The basic idea can be explained as  follows. If $G$ is a finite group then $\Z G$ is an order in the rational group algebra $\Q G$ and it is well known that
$\Q G = \prod_{i=1}^{n} M_{n_i}(D_i)$,
where each $D_i$ is a division algebra.
If $\O_i$ is an order in $D_i$, for each $i$, then $\O=\prod_{i=1}^{n} M_{n_i}(\O_i)$ is an order in $\Q G$ and the group of units $\U(\O)$ of $\O$ is commensurable with $\U(\Z G)$.
Recall that two subgroups of a given  group are said to be commensurable if they have a common subgroup that is of   finite index in both.
The group $\U(\O)$ is simply $\Pi_{i=1}^{n} \GL_{n_i}(\O_i)$, the direct product of the groups $\U(M_{n_i}(\O_i))=\GL_{n_i}(\O_i)$.
Moreover $\U(\mathcal{Z}(\O_i)) \times \SL_{n_i}(\O_i)$, the direct product of the central units in $\O_{i}$ and  the group consisting of the reduced norm one elements in $M_{n_{i}}(\O_{i})$, contains a subgroup of finite index isomorphic to a subgroup of finite index of $\GL_{n_i}(\O_i)$.
The group $\prod_{i=1}^{n} \U(\mathcal{Z}(\O_i))$ may be determined from Dirichlet's Unit Theorem and  it is commensurable with $\mathcal{Z} (\U (\Z G))$, the group of central units of $\Z G$.
For a large class of finite groups $G$ one can describe generators of a subgroup of finite index in  $\mathcal{Z} (\U (\Z G))$ \cite{JOdRV} (see also \cite{JdRV}).
Hence, up to commensurability, the problem of finding generators and relations for $\U(\Z G)$ reduces to finding a presentation of $\SL_{n_i}(\O_i)$ for every $1 \leq i \leq n$.
The congruence theorems allow to compute generators up to finite index for $\SL_{n_i}(\O_i)$ 
when   $n_i \geq 3$ (without any further restrictions) and also  for $n_i=2$ but then  provided   $D_{i}$ is neither a totally definite rational quaternion algebra, nor a quadratic imaginary extension of $\Q$ and also not $\Q$.
The case $n_i=1$ can also be dealt with in case  $D_i$ is commutative or when it is a  a totally definite quaternion algebra.
In case each $M_{n_i}(D_i)$ is of one of these types and if, moreover, $G$ does not have non-commutative fixed point free epimorphic images, then concrete generators, the so called  Bass units and bicyclic units,  for a subgroup of finite index in $\U (\Z G)$ have been determined \cite{SehgalBook,RS1991,1991RitterSehgal, RS1989, JespersLeal1993}.
A finite group $G$ is said to be of Kleinian type if each non-commutative simple factor $M_{n_i}(D_i)$ of $\Q G$ is a quaternion algebra over its centre, i.e. $n_{i}\leq 2$ and the natural image of $\SL_{n_i}(\O_i)$ in $\PSL_2(\C)$ (obtained by extending some embedding of the centre of $D_i$ in $\C$) is a Kleinian group.
A Kleinian group is a subgroup of $\PSL_2(\C)$ which is discrete for the natural topology, or equivalently,  its action on the 3-dimensional hyperbolic space via the Poincar\'{e} extension of the action by M\"{o}bius transformations is discontinuous \cite[Theorem~5.3.2]{Beardon}.
Poincar\'{e} introduced a technique to determine  presentations  for Kleinian groups via fundamental polyhedra (see e.g. \cite{Beardon,EGM,Maskit}).
An alternative method, due to Swan \cite{Swan}, gives a presentation from a connected open subset containing a fundamental domain.
Thus, if $G$ is a finite group of Kleinian type then, in theory, one can obtain a presentation of a group commensurable with $\U(\Z G)$.
In practice, it is difficult to execute  this procedure because it is usually hard to compute a fundamental Poincar\'{e} polyhedron of a Kleinian group.
Groups of Kleinian type have been classified in \cite{JPRRZ} and examples of how to find presentations of $\U(\Z G)$ for some groups of Kleinian type of small order are given in \cite{PRR} and \cite{PR}.
A generalization to group rings over commutative orders is given in \cite{OR}. Recently, in \cite{jesjurkie}, an algorithm is given to compute a fundamental domain, and hence generators, for subgroups of finite index in groups that are contained in the unit group of orders in quaternion algebras over quadratic imaginary extensions of $\Q$, in particular for Bianchi groups.
This work was in fact a generalization of \cite{corretall}.

In the investigations on determining presentations for $\U(\Z G)$ it is now natural to deal with finite groups
$G$ that are such that $\Q G$ has  simple components of  which the unit group
of some order does not necessarily act discontinuously on one hyperbolic space, but does so on a direct product of hyperbolic $2$ or $3$-spaces.
More precisely, one admits  simple components  $A$ of $\Q G$ that are quaternion algebras over a number field, say $F$.
Every field homomorphism $\sigma:F\rightarrow \C$ naturally extends to an isomorphism $\C\otimes_{\sigma(F)} A \cong M_2(\C)$ and thus to a homomorphism $\sigma:A\rightarrow M_2(\C)$, via $a\mapsto 1\otimes a$.
This homomorphism maps $\SL_1(A)$, the group consisting of reduced norm one elements, into $\SL_2(\C)$.
Composing this  with the action of $\SL_{2}(\C)$ on $\Hy^{3}$, by the Poincar\'{e} extension of M\"{o}bius transformations, yields   an action of $\SL_1(A)$ on $\Hy^3$.
Note that if $\sigma(F)\subseteq \R$ and $A$ is unramified in $\sigma$ (i.e. $\R\otimes_{\sigma(F)} A\cong M_2(\R)$), then $\sigma$ restricts to an action on $\Hy^2$.
Let $\sigma_1,\dots,\sigma_t$ be representatives (modulo complex conjugation) of the homomorphisms from $F$ to $\C$.
If $r$ is the number of real embeddings of $F$ on which $A$ is unramified and $s=t-r$, then the action of $\sigma_i$ on the $i$-th component gives an action of $\SL_1(A)$ on
	$$\Hy_{r,s}=\Hy^2\times \stackrel{(r)}{\dots} \times \Hy^2 \times \Hy^3\times \stackrel{(s)}{\dots} \times \Hy^3.$$
More precisely the action is given by
    $$a\cdot (x_1,\dots,x_t)=(\sigma_1(a)\cdot x_1,\dots,\sigma_t(a)\cdot x_t),$$
where $a\in \SL_1(A)$ and $x_1,\ldots ,x_r \in \Hy^{2}$, $x_{r+1},\ldots , x_{r+s} \in \Hy^3$.
If $\O$ is an order in $A$ then the image of $\O$ in  $M_2(\R)^r\times M_2(\C)^{(t-r)}$ is discrete.
This implies that the action of $\SL_1(\O)$ on $\Hy_{r,s}$ is discontinuous (see
Proposition~\ref{DiscreteDiscontinuous}).

This suggests the following program to obtain a presentation of $\U(\Z G)$ (more precisely of a group commensurable with $\U(\Z G)$) for the finite groups $G$ such that $\Q G$ is a direct product of fields and quaternion algebras.
(Equivalently, the degrees of the irreducible characters of $G$ are all either 1 or 2, or equivalently, by a result of Amitsur in \cite{amitsur}, $G$ is either abelian, or has an abelian subgroup of index 2 or $G/\mathcal{Z}(G)$ is elementary abelian of order 8.)
First determine the Wedderburn decomposition $\Q G \cong \prod_{i=1}^k A_i$.
For each $i=1,\dots,k$, let  $\O_i$ be an order in $A_i$ and calculate $\U(\mathcal{Z}(\O_i))$, using  for example the Dirichlet Unit Theorem.
Next, if   $A_i$ is  non-commutative and $\O =\O_{i}$ (in fact it is enough to consider the components which are not totally definite quaternion algebras) then calculate a fundamental domain for  the action of $\SL_1(\O)$ on $\Hy_{r,s}$ and determine from this a presentation of $\SL_1(\O)$.

For this program to be successful one needs to solve at least the following  problem.
\begin{enumerate}
\item[(1)] Does Poincar\'{e}'s Method remain valid  for discontinuous actions on $\Hy_{r,s}$?
\end{enumerate}
The following problem is one of the first issues to deal with in order to answer the question.
\begin{enumerate}
\item[(2)] Determine methods to calculate fundamental domains effectively.
\end{enumerate}

Once this is done, the following problem arises.
\begin{enumerate}
\item[(3)] The sides of the well-known fundamental polyhedra in $\Hy^2$ and $\Hy^3$ are geodesic hyperplanes i.e. lines and circles (respectively, planes and spheres) orthogonal to the border. How should one define the ``sides'' of some potential fundamental domain in $\Hy_{r,s}$?
\end{enumerate}

In this paper we show that the three problems can be controlled  in case  the considered group acts  on a direct product of two copies of $\Hy^2$. More precisely, we complete the outlined program for the Hilbert modular group $\PSL_2(R)$, where $R$ is the ring of integers in $\Q(\sqrt{d})$, with $d>0$, and  $R$ is a principal ideal domain.
The group  $\PSL_2(R)$ acts  discontinuously   on $\HTwo$, the direct product of $2$ copies of hyperbolic $2$-space.
Hence, this group constitutes the ``easiest'' case of the outlined program and as such can be considered a relevant test case. For the construction  of a fundamental domain
we will make use of some   ideas in \cite{HarveyMathComp} and \cite{HarveyProc}. 
The construction is based on the classical construction of a Ford fundamental domain. For more details we refer the reader to  
\cite{Voight} or to  \cite[Section 9.5]{Beardon}, where the author calls such domains Generalized Dirichlet domains. The largest part of the paper is devoted to set up the machinery needed to prove an  extension of  the presentation part of Poincar\'e's Polyhedron theorem so that we can  extract a presentation from this domain.
The assumption that $R$ is a principal ideal domain  allows us to get a description of the fundamental domain $\F$ in terms of finitely many varieties.
Note that this restriction  is essentially only used in Lemma~\ref{minh}, Theorem~\ref{FiniteSides}, Lemma~\ref{esscompact} and Corollary~\ref{generators}.

The  method outlined allows computing a description of the unit group $\U (\Z G)$ for several classes of finite groups $G$. As a matter of example, we mention that this now easily can be done for the dihedral groups $D_{10}$ and $D_{16}$ of order 10 and 16 respectively, and the quaternion group $Q_{32}$ of order 32.  The group  $\U(\Z D_{10})$ is commensurable with $\SL_2\left(\Z[\frac{1+\sqrt{5}}{2}]\right)$, the group
$\U(\Z D_{16})$ is commensurable with a direct product of a non-abelian free group and the group $\SL_2(\Z[\sqrt{2}])$. The group $\U (\Z Q_{32})$  is commensurable with a direct product of $\U( \Z D_{16})$ and an abelian group. Corollary~\ref{generators} gives generators of the groups $\SL_2(\Z[\frac{1+\sqrt{5}}{2}])$ and $\SL_2(\Z[\sqrt{2}])$ and the   algorithm  described in Theorem~\ref{Poincarerelations} yields a presentation of these groups. If one does not want to work up to commensurability then  the Reidemeister-Schreier method allows to go down to the exact subgroup of the previous groups needed for the computation of the corresponding unit groups. In the case of $\U(\Z D_{16})$ this exact subgroup is described in \cite{JesPar}.

An alternative method for computing a presentation for the group of units of an order in a semi-simple algebra over $\Q$ was recently given by Coulangeon, Nebe, Braun and Sch\"onnenbeck in \cite{CNBS}. The authors present a powerful algorithm that is based on a generalisation  of  Vorono\"{\i}'s algorithm for computing perfect forms and is combined with Bass-Serre theory.
The method differs essentially from the one presented here, as the authors use an action on a Euclidean space. 
To illustrate their algorithm, several computations are carried out completely in \cite{CNBS}; including orders in division algebras of degree $3$.

As mentioned above, in this paper we focus on groups that act discontinuously on direct products of two hyperbolic spaces
and we  extend the  presentation part of Poincar\'e's Polyhedron  Theorem into this context.
Developing such  new methods are relevant in the bigger scheme of discovering new generic constructions of units that generate large subgroups in the unit group of an  order of a rational division algebra and hence solving completely the problem of describing finitely many generic generators for a subgroup of finite index in $\U (\Z G)$ for any finite group (hence without any restriction on the rational group algebra $\Q G$).

The paper is organized as follows.
In Section 2, we give some background on discontinuous group actions on direct products of copies of $\Hy^2$ and $\Hy^3$ and introduce some notation on the group of our interest, namely $\PSL_2(R)$ with $R$ the ring of integers of a real quadratic field, and its action on $\HTwo$.
In Section 3, we give a description of a fundamental domain $\F$ for this action, see Theorem~\ref{FiniteSides}.
In Section 4, we prove some topological lemmas about the fundamental domain and they will be used in the two following sections.
Finally, in Section 5, we describe the sides of the fundamental domain $\F$. In Theorem~\ref{Poincare}, we generalize the generating part of Poincar\'e's Polyhedron Theorem to the case of $\PSL_2(R)$ acting on $\HTwo$ and in Corollary~\ref{generators}, we give an effective description of the generators of $\PSL_2(R)$.
In Section 6, we describe the edges of $\F$ and in Theorem~\ref{Poincarerelations}, we generalize the presentation part of Poincar\'e's Polyhedron Theorem. Note that Theorem~\ref{Poincare} and Theorem~\ref{Poincarerelations} could potentially have been deduced from \cite[Corollary of Theorem 2]{Macbeath}.
However, in order to get an effective set of generators and relations as an application of this theorem one needs  all the lemmas in Section 5 and 6  anyway. Moreover, we think that the proofs of Theorem~\ref{Poincare} and Theorem~\ref{Poincarerelations}, presented in this paper, are more intuitive from a geometric point of view. Furthermore they  are algorithmic in nature.
For the convenience of the reader we include an appendix on some needed algebraic topological results.

\section{Background}\label{SectionBackground}
The following definitions are mostly taken from \cite{Beardon, Rat}. Let $\widehat{\C}= \C \cup \lbrace \infty \rbrace$ and let $\Hy^n$ denote the upper half Poincar\'e model of the hyperbolic space of dimension $n$. We consider the group $\SL_2(\C)$ acting on $\widehat{\C}$ by M\"obius transformations. We also let $\SL_2(\C)$ act by orientation-preserving isometries on $\Hy^3$, by the Poincar\'e extension of M\"obius transformations. The hyperbolic plane $\Hy^2$ is invariant under the action of $\SL_2(\R)$. This defines an isomorphism between $\PSL_2(\C)$ (respectively $\PSL_2(\R)$) and the group of orientation-preserving isometries of $\Hy^3$ (respectively $\Hy^2$).

Consider for non negative integers $r$ and $s$ the metric space $\Hy_{r,s}=\Hy^2 \times \stackrel{(r)}{\ldots}\times \Hy^2 \times \Hy^3 \times \stackrel{(s)}{\ldots}  \times \Hy^3$ and the group $\mathcal{G}_{r,s}=\SL_2(\R)\times \stackrel{(r)}{\ldots} \times \SL_2(\R) \times \SL_2(\C) \times \stackrel{(s)}{\ldots} \times \SL_2(\C)$, a direct product of $r$ copies of $\SL_2(\R)$ and $s$ copies of $\SL_2(\C)$.
An element $g \in \mathcal{G}_{r,s}$ will be written as an $(r+s)$-tuple $(g^{(1)}, \ldots , g^{(r+s)})$ and elements of $\Hy_{r,s}$ will be written as  $(r+s)$-tuples, say
$Z=(Z_{1}, \ldots , Z_{r+s})$. The metric $\rho$ on this space is given
by
\begin{equation*}
\rho(X,Y)^2 = \sum_{i=1}^{n} \rho^{(i)}(X_{i},Y_{i})^2,
\end{equation*}
where $\rho^{i}$ is the hyperbolic metric on  the hyperbolic 2 or 3-space (see \cite[Paragraph 2]{Maass}).
The action of $\SL_2(\R)$ and $\SL_2(\C)$ on $\Hy^2$ and $\Hy^3$ induces componentwise an action of $\mathcal{G}_{r,s}$ by isometries on $\Hy_{r,s}$:
	$$g\cdot Z = (g^{(1)}\cdot Z_1, \dots, g^{(r+s)}\cdot Z_{r+s}).$$

We consider $M_n(\C)$ as an Euclidean $2n^2$-dimensional real space. This induces a structure of topological group on $\mathcal{G}_{r,s}$.
Let $G$ be a subgroup of $\mathcal{G}_{r,s}$. One says that $G$ is \emph{discrete} if it is discrete with respect to the induced topology. One says that $G$ acts \emph{discontinuously} on $\Hy_{r,s}$ if and only if, for every compact subset $K$ of $\Hy_{r,s}$,  the set  $\{ g\in G \; : \; g(K) \cap K \neq  \emptyset \}  $ is finite.

The following proposition is well known and extends the fact that a subgroup of
$\SL_2(\C)$ is discrete if and only if its action on
$\Hy^3$ is discontinuous.

\begin{proposition}\cite[Theorem 5.3.5.]{Rat}\label{DiscreteDiscontinuous}
A subgroup of $\mathcal{G}_{r,s}$ is discrete if and only if it acts discontinuously on $\Hy_{r,s}$.
\end{proposition}


Let $G$ be a group of isometries of $\Hy_{r,s}$. A \emph{fundamental domain} of $G$ is a subset $\F$ of $\Hy_{r,s}$ whose boundary has Lebesgue measure $0$, such that $\Hy_{r,s}=\bigcup_{g\in G} g(\F)$ and if $1\ne g\in G$ and $Z\in \Hy_{r,s}$ then $\{Z,g(Z)\}$ is not contained in the interior of $\F$. The different sets $g(\F)$, for $g \in G$, are called \emph{tiles} given by $G$ and $\F$ and the set of tiles $\te=\{g(\F)\ :\ g\in G\}$ is called the \emph{tessellation} given by $G$ and $\F$.

In this paper we study a presentation for the special linear group of degree $2$ of the ring of integers of real quadratic extensions. So throughout $k$ is a square-free positive integer greater than 1 and
    \begin{equation}\label{omega}
     \omega=\frac{1+\sqrt{k}}{k_0}, \text{ with } k_0=\left\{\begin{array}{ll} 1, & \text{if } k \not\equiv 1 \mod  4; \\ 2, & \text{if } k \equiv 1 \mod  4.
    \end{array}\right.
    \end{equation}
Moreover, $K=\Q\left(\sqrt{k}\right)$ and $R=\Z[\omega]$, the ring of integers of $K$.
For $\alpha\in K$, let $\alpha'$ denote the algebraic conjugate of $\alpha$.
Then $\omega'=\frac{1-\sqrt{k}}{k_0}$ and if $\alpha=\alpha_0+\alpha_1 \omega$ with $\alpha_0,\alpha_1\in \Q$, then $\alpha'=\alpha_0+\alpha_1 \omega'$.
Let $N$ denote the norm map of $K$ over $\Q$ and $\epsilon_0$ denote the fundamental unit of $K$, i.e. $\epsilon_0$ is the smallest unit $\epsilon_0$ of $R$ greater than 1.
Then $\U(R)=\{\alpha\in R : N(\alpha)=\pm 1\}=\pm \GEN{\epsilon_0}$, by Dirichlet's Unit Theorem.

We consider $\SL_2(R)$ as a discrete subgroup of $\SL_2(\R)\times \SL_2(\R)$ by identifying a matrix $A\in \SL_2(R)$ with the pair $(A,A')$, where $A'$ is the result of applying the algebraic conjugate $'$ in each entry of $A$.
Thus we consider $\SL_2(R)$ acting on $\HTwo$.
More precisely, if $\gamma=\pmatriz{{cc} a & b \\ c & d}\in \SL_2(K)$, $Z=(x_1+y_1i,x_2+y_2i)$ and $\gamma(Z)=(\hat{x}_1+\hat{y}_1i,\hat{x}_2+\hat{y}_2i)$ then a straightforward calculation yields
    \begin{equation}\label{chapeau1}
    \hat{x}_1 = \frac{(ax_1+b)(cx_1+d)+acy_1^2}{(cx_1+d)^2+c^2y_1^2}, \quad \hat{y}_1 = \frac{y_1}{(cx_1+d)^2+c^2y_1^2}
    \end{equation}
and
    \begin{equation}\label{chapeau2}
    \hat{x}_2 = \frac{(a'x_2+b')(c'x_2+d')+a'c'y_2^2}{(c'x_2+d')^2+c'^2y_2^2}, \quad \hat{y}_2 = \frac{y_2}{(c'x_2+d')^2+c'^2y_2^2}.
    \end{equation}
If $Z=(Z_1,Z_2)\in \Hy^2\times \Hy^2$ then we write
\begin{equation}
Z_j=x_j+y_ji, \quad \text{where } x_j,y_j\in \R \text{ and } y_j>0 \quad (j=1,2).
\end{equation}
Then the four tuples $(x_1,x_2,y_1,y_2)\in \R^2\times (\R^+)^2$ form a system of coordinates of elements of $\HTwo$.

If $Z=(x_1+y_1i,x_2+y_2i)$ then we use the following notation
	$$\Vert cZ+d \Vert = ((cx_1+d)^2+c^2y_1^2)((c'x_2+d')^2+{c'}^2y_2^2).$$

We introduce another system of coordinates $(s_1,s_2,r,h)\in \R^2\times (\R^+)^2$ by setting
	\begin{equation}\label{Rtox}
	x_1  =  s_1+s_2\omega, \quad x_2  =  s_1+s_2\omega', \quad y_1^2=\frac{h}{r}, \quad y_2^2= rh
	\end{equation}
or equivalently
\begin{equation} \label{formulaRi}
s_1 = \frac{\omega' x_1 - \omega x_2}{\omega'-\omega}, \quad
s_2 = \frac{x_1-x_2}{\omega-\omega'}, \quad
r = \frac{y_1}{y_2}, \quad h = y_1 y_2.
\end{equation}
So, each element $Z$ of $\HTwo$ is represented by either $(x_1,x_2,y_1,y_2) \in \R^2 \times (\R^{+})^2$, or $(s_1,s_2,r,h) \in \R^2 \times (\R^{+})^2$, or $(x_1,x_2,r,h)\in \R^2 \times (\R^+)^2$.
Then the norm, ratio and height of $Z$ are defined respectively by
$$\Vert Z \Vert = (x_1^2+y_1^2)(x_2^2+y_2^2), \quad r(Z)=r=\frac{y_1}{y_2}, \quad h(Z) = h=y_1y_2.$$

Using (\ref{chapeau1}) and (\ref{chapeau2}), we get
\begin{equation}\label{Height}
\quad h(\gamma(Z)) = \frac{h(Z)}{\Vert cZ+d \Vert}.
\end{equation}

Let
\begin{equation}\label{defgamma}
\Gamma=\PSL_2(R),
\end{equation}
the Hilbert modular group, and we consider it as a discontinuous group of isometries of $\HTwo$.
Let $\Gamma_{\infty}$ denote the stabilizer of $\infty$ by the action of $\Gamma$ (on $\widehat{\C}$).
The elements of $\Gamma_{\infty}$ are represented by the matrices $\pmatriz{{cc} \epsilon_0^m & b \\ 0 & \epsilon_0^{-m}}$ with $m\in \Z$ and $b\in R$.
Abusing notation, if $\gamma\in \Gamma$ is represented by $\pmatriz{{cc} a & b \\ c & d}$ then we simply write $\gamma=\pmatriz{{cc} a & b \\ c & d}$.
Hence, (\ref{chapeau1}), (\ref{chapeau2}) and (\ref{Height}) imply
\begin{equation}\label{HeightInfinity}
\text{if } \gamma = \begin{pmatrix}
 \epsilon_0^m & b \\ 0 & \epsilon_0^{-m} \end{pmatrix} \in \Gamma_{\infty} \text{ then } h(\gamma(Z)) = h(Z) \text{ and } r(\gamma(Z))=\epsilon_0^{4m} r(Z).
\end{equation}

The second rows of the elements of $\Gamma$ form the following set
$$\S=\{(c,d)\in R^2 : cR+dR=R\}.$$

\section{Fundamental Domain}

In this section we compute a fundamental domain for the group $\Gamma=\PSL_2(R)$ (\ref{defgamma}) via  methods analogue to those used for computing a Ford fundamental domain of a discrete group acting on the hyperbolic $2$-space. This part of our work is based on the ideas of H. Cohn in \cite{HarveyProc,HarveyMathComp}.

We start introducing the following subsets of $\HTwo$, expressed in the $(s_1,s_2,r,h)$ coordinates:
\begin{eqnarray*} 
V_{i}^{+,\ge} & = & \left\{ (s_1,s_2,r,h) \in \HTwo \; : \; s_i\ge \frac{1}{2}\right\} , \\
V_{i}^{-,\ge} & = & \left\{ (s_1,s_2,r,h) \in  \HTwo\; : \; s_i\ge -\frac{1}{2}\right\},
\end{eqnarray*}
for $i=1,2$ and
\begin{eqnarray*}
V_{3}^{+,\ge} & = & \left\{ (s_1,s_2,r,h)\ \in  \HTwo; : \; r\ge \epsilon_0^{2} \right\} , \\
V_{3}^{-,\ge} & = & \left\{ (s_1,s_2,r,h)\in  \HTwo \; : \; r\ge \epsilon_0^{-2}  \right\}.
\end{eqnarray*}
Moreover, for every $c,d\in \S$ with $c\ne 0$, let
\begin{equation*}
V_{c,d}^{\ge} = \lbrace Z \in \HTwo : \Vert cZ+d \Vert \ge 1\rbrace.
\end{equation*}
We also define the sets $V_i^{\pm,\le}$ and $V_{c,d}^{\le}$ (respectively, $V_i^{\pm}$ and $V_{c,d}$) by replacing $\ge$ by $\le$ (respectively, $=$) in the previous definitions.

For every $(c,d) \in \S$ with $c \neq 0$ and $\gamma = \begin{pmatrix} a & b \\ c & d \end{pmatrix} \in \Gamma$ we have
\begin{equation}\label{equationfraction}
\Vert -c\gamma(Z) + a \Vert = \frac{1}{\Vert cZ+d\Vert }
\end{equation}
and therefore
\begin{equation*}
 \gamma\left(V_{c,d}^{\ge}\right)=V_{-c,a}^{\le}.
\end{equation*}

We define
	\begin{eqnarray*}
	 \mathcal{F}_{\infty}&=&V_1^{+,\le}\cap V_1^{-,\ge} \cap V_2^{+,\le}\cap V_2^{-,\ge} \cap V_3^{+,\le}\cap V_3^{-,\ge} \\
											 &=& \{(s_1,s_2,r,h)\in \R^2\times (\R^+)^2 : \vert s_1 \vert, \vert s_2 \vert \le \frac{1}{2}, \epsilon_0^{-2}\le r \le \epsilon_0^2\}; \\
   \mathcal{F}_0 &=& \underset{(c,d)\in \S}{\bigcap} V^{\ge}_{c,d} = \{Z\in \HTwo : \Vert cZ+d \Vert \ge 1 \text{ for all } (c,d)\in \S\}; \\
	\mathcal{F} &=& \mathcal{F}_{\infty} \cap \mathcal{F}_0.
	\end{eqnarray*}
	
\begin{lemma}\label{FundDomainInfinity}
$\mathcal{F}_{\infty}$ is a fundamental domain of $\Gamma_{\infty}$.
\end{lemma}

\begin{proof}
We first prove that if $1\ne \gamma \in \Gamma_{\infty}$ and $Z=(Z_1,Z_2)=(x_1+y_1i,x_2+y_2i)=(s_1,s_2,r,h)$ then $\{ Z,\gamma (Z)\}$ cannot be contained in the interior of $\mathcal{F}_{\infty}$.
Let
\begin{equation*}
\gamma=\begin{pmatrix} \epsilon_0^m & b \\ 0 & \epsilon_0 ^{-m} \end{pmatrix}.
\end{equation*}
By (\ref{HeightInfinity}), $r(\gamma(Z))=\epsilon_0^{4m}r$.
If $Z$ and $\gamma(Z)$ belong to the interior of $\Gamma_{\infty}$ then $\epsilon_0^{-2} < r,\epsilon_0^{4m}r < \epsilon_0^2$ and therefore $m=0$.
Therefore the transformation $\gamma$ is simply a translation by the parameter $(b,b')$ with $b \in R$.
Now $b=b_1+b_2\omega$, with $b_1,b_2\in \Z$.
As $Z$ and $\gamma(Z)$
belong to the interior of $\mathcal{F}_{\infty}$, then $\vert s_1 \vert ,\vert s_2 \vert , \vert s_1+b_1 \vert,
\vert s_2+b_2 \vert < \frac{1}{2}$. Thus $b=0$ and hence $\gamma=1$, as desired.

Let $Z=(x_1+y_1\omega,x_2+y_2\omega')$ be an arbitrary point in  $ \Hy^2\times \Hy^2$, with $x_{i},y_{i}\in \R$.
We will show that there exists $\gamma \in \Gamma_{\infty}$ such that $\gamma(Z) \in \mathcal{F}_{\infty}$. As $\epsilon_0 > 1$, $\lim\limits_{n\rightarrow +\infty}\epsilon_0^n = \infty$ and $\lim\limits_{n\rightarrow -\infty}\epsilon_0^n = 0$ and hence there exists $n \in \Z$ such that
\begin{equation*}
\epsilon_0^{4n-2} \leq r \leq \epsilon_0^{4n+2}.
\end{equation*}
Let $\gamma=\begin{pmatrix}
\epsilon_0^{-n} & 0 \\
0 & \epsilon_0^{n} \end{pmatrix}$.
By (\ref{HeightInfinity}), $\epsilon_0^{-2}\le r(\gamma(Z))=\epsilon_0^{-4n}r \le \epsilon_0^2$.
So we may assume that $\epsilon_0^{-2}\le r \le \epsilon_0^2$.
Now consider $\gamma=\pmatriz{{cc} 1 & b \\ 0 & 1}$, with $b=b_1+b_2\omega$, $b_1,b_2\in \Z$ and
$\vert x_i+b_i \vert \le \frac{1}{2}$. Again by (\ref{HeightInfinity}), the $(s_1,s_2,r,h)$ coordinates of $\gamma(Z)$ are $(s_1+b_1,s_2+b_2,r,h)$ and hence $\gamma(Z)\in \F_{\infty}$, as desired.

Moreover the boundary of $\F_{\infty}$ is included in $\bigcup_{i=1}^3 V_i^+ \cup \bigcup_{i=1}^3 V_i^{-}$.
This clearly has Lebesgue measure $0$.
\end{proof}

For a subset $C$ of $\HTwo$ and $\mu>0$ let
	\begin{equation}\label{VsDefinition}
	 V_{C,\mu}=\{(c,d)\in R\times R : \Vert cZ+d\Vert \le \mu \text{ for some } Z\in C\}.
	\end{equation}

We say that a subset $C$ of $\HTwo$ is \emph{hyperbolically bounded} if it is bounded in the metric of $\HTwo$ or equivalently if it is bounded in the Euclidean metric and there is a positive number $\epsilon$ such that $y_1,y_2>\epsilon$ for every $(x_1+y_1 i,x_2+y_2i)\in C$.

\begin{lemma}\label{BoundedFinite}
Let $C$ be a  hyperbolically bounded  subset of $\HTwo$ and let $\mu>0$.
Then there are finitely many elements $(c_1,d_1),\dots,(c_k,d_k)$ of $R\times R$ such that
$V_{C,\mu}=\{(uc_i,ud_i) :
i=1,\dots,k, u\in \U(R)\}$.
\end{lemma}

\begin{proof}
As $\{d\in R : N(d)\le \mu \}$ is finite up to units, it is clear that $V_{C,\mu}$ has finitely many elements of the form $(0,d)$ with $d\in R$.
As $C$ is hyperbolically bounded there is $s>0$ such that $h(Z)\ge s$ for every $Z\in C$.
As the number of elements of $R$ of a given norm is finite modulo units, there are $c_1,\dots,c_k \in
R\setminus \{0\}$ such that if $c\in R\setminus \{0\}$ with $N(c)^2\le \frac{\mu}{s^2}$ then $c=uc_i$ for some $i=1,\dots,k$ and some
$u\in \U(R)$. Moreover, for every $i=1,\dots,k$, the set $C_i=\{(x,y) \in \R^{2} \; : \; (c_ix_1+x)^2c_i'^2
y_2^2 + (c_i'x_2+y)^2 c_i^2 y_1^2 \le r \text{ for some  } Z=(x_1+y_1i,x_2+y_2i)\in C\}$ is a bounded subset of $\R^2$. As $\{(d,d'):d\in R\}$ is a discrete subset of $\R^2$, there are
$e_{i1},\dots,e_{ij_i} \in R$ such that if $d\in R$ and $(d,d')\in C_i$ then
$d=e_{il_{i}}$ for some $1\leq l_{i}\leq j_{i}$.
Assume that $\Vert cZ +d \Vert\le \mu$ for some $Z=(x_1+y_1i,x_2+y_2i)\in C$.
Then $N(c)^2s^2 \le N(c)^2h(Z)^2 = c^2{c'}^2y_1^2y_2^2 \le ((cx_1+d)^2+c^2y_1^2)((c'x_2+d')^2+c'^2y_2^2)=
\Vert cZ+d \Vert \le \mu$. Hence $c=uc_i$ for some $i=1,\dots,k$ and $u\in \U(R)$.
Moreover, as $N(u)=uu'= \pm 1$, we have
$$(c_ix_1+u^{-1}d)^2c_i'^2 y_2^2 + (c_i'x_2+(u^{-1}d)')^2 c_i^2 y_1^2=
(cx_1+d)^2c'^2 y_2^2 +
(c'x_2+d')^2 c^2 y_1^2 \le \Vert cZ+d \Vert \le \mu.$$
Therefore $(u^{-1} d,(u^{-1} d)')\in C_i$. Thus $u^{-1} d=e_{il_{i}}$ for some $1\leq l_{i}\leq j_i$.
This proves that there are finitely many elements $(c_1,d_1),\dots,(c_k,d_k)\in (R\setminus \{0\})\times R$ such
that $V_{C,\mu}$ is contained in $\{u(c_i,d_i) : i=1,\dots,k, u\in \U(R)\}$.
Note that if $u\in \U(R)$ then $\Vert uc_iZ+ud_i\Vert = \Vert c_iZ+d_i \Vert$. Hence, the result
follows.
\end{proof}

\begin{lemma}\label{maxhorbit}
$\F_0 = \lbrace Z \in \HTwo :\ Z \textrm{ has maximal height in its } \Gamma\textrm{-orbit} \rbrace$.
\end{lemma}

\begin{proof}
We first claim that for a fixed $Z\in \HTwo$, the set $\{\Vert cZ+d\Vert : (c,d)\in \S \}$ has a minimum.
Indeed, let $\pi=\Vert Z \Vert$. Clearly $V_{\{Z\},\pi}\cap \S \neq \emptyset$ as it contains
$(1,0)$. By Lemma~\ref{BoundedFinite}, $V_{\{Z\},\pi}\cap \S=\{u(c_i,d_i) : i=1,\dots,k, u\in \U(R)\}$ for
some $(c_1,d_1),\dots,(c_k,d_k)\in \S$. Let $0\neq m = \min \{\Vert c_i Z + d_i \Vert : i=1,\dots, k\}$.
If $\Vert cZ+d \Vert < m\le \pi$, with $(c,d)\in \S$ then $(c,d)=u(c_i,d_i)$ for some $i$ and
some $u\in \U(R)$. Then $m>\Vert cZ +d \Vert = N(u)^2 \Vert c_i Z + d_i \Vert \geq  m$, a contradiction.
Hence the claim follows.
Consequently, by (\ref{Height}), the $\Gamma$-orbit of $Z$ has an element of maximum height.

Consider a $\Gamma$-orbit and let $Z$ be an element in this orbit with maximal height. Hence for every $g \in \h$, $h(g(Z)) \leq h(Z)$ and hence by (\ref{Height}), $\Vert cZ + d \Vert \geq 1$, for every $(c,d) \in \S$. Thus $Z\in \mathcal{F}_0$.

To prove the other inclusion, let $Z \in \F_0$. Then $\Vert cZ +d \Vert \geq 1$, for every $(c,d) \in \S$ and hence for every $g \in \h$, $h(g(Z)) \leq h(Z)$. Thus every element of $\F_0$ reaches the maximum height in its orbit.
\end{proof}

\begin{lemma} \cite{HarveyProc,HarveyMathComp} \label{funddomain}
$\mathcal{F}$ is fundamental domain for $\Gamma$.
\end{lemma}

\begin{proof}
We first show that $\mathcal{F}$ contains a point of every $\Gamma$-orbit. To prove this, consider an orbit and let $Z$ be an element in this orbit with maximal height.
By Lemma~\ref{FundDomainInfinity}, there is $W\in \F_{\infty}$ and $g\in \h_{\infty}$ such that  $W=g(Z)$.
By (\ref{HeightInfinity}), $h(Z)=h(W)$ and hence we may assume that $Z\in \F_{\infty}$. Thus,
by Lemma~\ref{maxhorbit},
$Z\in \mathcal{F}_0$ and so $Z \in \F$.

Now we will show that if two points of the same orbit are in $\mathcal{F}$ then they necessarily are on the border of $\mathcal{F}$. Suppose that $Z,\gamma(Z)\in \F$, for some $1 \neq \gamma \in \h$. If $\gamma\in \h_{\infty}$ then, by Lemma~\ref{FundDomainInfinity}, $Z$ belongs to the border of $\F_{\infty}$ and hence it belongs to the border of $\F$. Otherwise $\gamma= \pmatriz{{cc} a&b\\c&d}$ with $c\ne 0$ and, by (\ref{Height}) we have
\begin{equation*}
h(Z) = h(\gamma(Z))=\frac{h(Z)}{\Vert cZ+d \Vert}.
\end{equation*}
Therefore $\Vert cZ+d \Vert=1$ and thus $Z$ lies on the border of $\mathcal{F}_0$ and thus on the border of $\mathcal{F}$.

Finally the boundary of $\F$ is contained in $\bigcup_{i=1}^3 V_i^{+} \cup \bigcup_{i=1}^3 V_i^{-} \cup \bigcup_{(c,d) \in \S} V_{c,d}$ with $c \neq 0$. As $\S$ is countable, the boundary of $\F$ has measure $0$.
\end{proof}

Let
	$$B= \left\{(s_1,s_2,r) \; : \; \vert s_1 \vert, \vert s_2 \vert \le \frac{1}{2},\; \epsilon_0^{-2} \le r \le \epsilon_0^2\right\}$$
Observe that
	$$\mathcal{F}_{\infty}=B \times \R^+.$$
We can think of $\F$ as the region above a ``floor'', which is given by the sets $V_{c,d}$ for $(c,d) \in \S$ and limited by ``six walls'', which are given by the sets $V_i^{\pm}$ for $i=1,2,3$.

We now give an alternative description of $\F$ on which the ``floor'' is given by the graph of a function $h_0$ defined on $B$.

For each $c,d\in R$ we define the function $f_{c,d}:\R^4 \rightarrow \R$ by
	$$f_{c,d}(x_1,x_2,y_1,y_2)  =  \left[(cx_1+d)^2+c^2y_1^2 \right] \left[ ({c}' x_2+d')^2+{c}'^2y_2^2 \right].$$
Observe that if $Z\in \HTwo$ then $f_{c,d}(Z)=\Vert cZ+d \Vert$ and if $d\in \U(R)$ then $f_{0,d}(Z)=1$.
Thus
	\begin{equation}\label{Ff}
	 \mathcal{F}_0 = \{Z\in \HTwo :f_{c,d}(Z)\ge 1, \text{ for all } (c,d)\in \S\}.
	\end{equation}
We use the mixed coordinate system described in Section~\ref{SectionBackground} and write
	$$f_{c,d}(x_1,x_2,r,h)=\left[ (cx_1+d)(c'x_2+d') \right]^2+\left[(cx_1+d)^2c'^2r+(c'x_2+d')^2\frac{c^2}{r} \right]v + N(c)^2v^2.$$
If $c,d\in R$, $(x_1,x_2,r) \in \R^2\times \R^+$ and $v\in \R$, then we define
\begin{equation}\label{function}
f_{c,d,x_1,x_2,r}(v)=\left[ (cx_1+d)(c'x_2+d') \right]^2+\left[(cx_1+d)^2c'^2r+(c'x_2+d')^2\frac{c^2}{r} \right]v + N(c)^2v^2.
\end{equation}
and if $c\ne 0$ then we set
    \begin{eqnarray*}
    &&\hspace{-1.2cm} h_1(c,d,x_1,x_2,r)= \\
    &&
    \sqrt{\frac{1}{N(c)^2}+\frac{1}{4}\left[\left(x_1+\frac{d}{c}\right)^2r-\left(x_2+\frac{d'}{c'}\right)^2\frac{1}{r} \right]^2}-
    \frac{1}{2}\left[\left(x_1+\frac{d}{c}\right)^2r+\left(x_2+\frac{d'}{c'}\right)^2\frac{1}{r} \right],
    \end{eqnarray*}
If $u\in \U(R)$ then the following statements hold (throughout the paper we will use these without explicit reference):
\begin{eqnarray}
f_{uc,ud}&=&f_{c,d}, \label{independent-unit1}\\
h_{1}(uc,ud,x_1,x_2,r)&=&h_{1}(c,d,x_1,x_2,r), \label{independent-unit2}\\
(c,d)\in \S & \Leftrightarrow & (uc,ud)\in \S \label{independent-unit3}\\
V_{uc,ud}^{\ge}&=&V_{c,d}^{\ge}, \label{independent-unit4} \\
V_{uc,ud}&=&V_{c,d}. \label{independent-unit5}
\end{eqnarray}

Set $\mathcal{C}=\lbrace (x_1,x_2) \in \R^2 \ :\ -1< (cx_1+d)(c'x_2+d')< 1 \rbrace$.
An easy calculation shows that if $c\ne 0$ then
    $$V_{c,d}=\{(x_1,x_2,r,h_1(c,d,x_1,x_2,r)): (x_1,x_2,r)\in \mathcal{C} \times \R^+\}.$$
As $\mathcal{C}$ is path-connected, $V_{c,d}$ is path-connected.
Moreover, if $f=f_{c,d,x_1,x_2,r}$ then
\begin{eqnarray*}
f'(h)& = & \left[(cx_1+d)^2c'^2r+(c'x_2+d')^2\frac{c^2}{r} \right] + 2N(c)^2h.
\end{eqnarray*}
and hence $f$ is strictly increasing on $[\min(0,h_1),\infty)$.

For $(s_1,s_2,r)\in B$ let
    \begin{eqnarray*}
    h_0(s_1,s_2,r) &=& \sup \{ h_1>0 : f_{c,d,x_1,x_2,r}(h_1)=1, \text{ for some } (c,d)\in \S, \text{ with } c\ne 0\} ,
    \end{eqnarray*}
where we understand that the supremum of the empty set is $0$.
If $(c,d)\in \S$, $c\ne 0$, $h_1>0$ and $f_{c,d,x_1,x_2,r}(h_1)=1$ then $h_1^2 \le \frac{1}{N(c)^2}\le 1$.
Hence the supremum defining $h_0(s_1,s_2,r)$ exists and
\begin{equation}\label{h0atmost1}
h_0(s_1,s_2,r) \leq 1.
\end{equation}

By (\ref{Ff}) and the monotonicity $f_{c,d,x_1,x_2,r}$ we have
\begin{eqnarray}
\mathcal{F}_0 & = & \{(s_1,s_2,r,h)\in \R^2\times (\R^+)^2 : 
	h\ge h_0(s_1,s_2,r)\} \textrm{ and } \label{fzero}\\
\mathcal{F} & = & \{(s_1,s_2,r,h)\in \R^2\times (\R^+)^2 : (s_1,s_2,r) \in B \quad \text{and} \quad
	h\ge h_0(s_1,s_2,r)\}. \label{fnotzero}
\end{eqnarray}

From (\ref{h0atmost1}) and (\ref{fnotzero}) the following lemmas easily follow.

\begin{lemma}\label{secureheight}
If $Z=(s_1,s_2,r,h) \in \HTwo$ with $\vert s_1 \vert \leq \frac{1}{2}$, $\vert s_2 \vert \leq \frac{1}{2}$, $\epsilon_0^{-2} \leq r \leq \epsilon_0^2$ and $h \geq 1$, then $Z \in \F$. Moreover, if the inequalities are strict then $Z \in \F^{\circ}$.
\end{lemma}

\begin{lemma}\label{fpathconnected}
$\F$ is path-connected.
\end{lemma}

In order to have a finite procedure to calculate the fundamental domain $\F$ of Lemma~\ref{funddomain} we need to replace $\S$ in the definition of $\F$ by a suitable finite set.
In our next result we obtain this for $R$ a principal ideal domain (PID, for short). For that we need the following lemma, which is proved in \cite[Paragraph 5]{HarveyProc}.

\begin{lemma}\label{minh}
 If $R$ is a PID and $(s_1,s_2,r,h)\in \mathcal{F}_0$ then $h>\frac{k_0^2}{2k}$.
\end{lemma}


\begin{theorem}\label{FiniteSides}
Let $k$ be a square-free integer greater than 1.
Let $k_0$ and $\omega$ be as in (\ref{omega}), $R=\Z[\omega]$, $\Gamma=\PSL_2(R)$ and $\F$ the fundamental domain of $\Gamma$ given in Lemma~\ref{funddomain}.
Let $\S_{1}$ be a set of representatives, up to multiplication by units in $R$, of the couples $(c,d)\in R^{2}$ satisfying the following conditions:
\begin{equation}\label{Conditionscd}
\begin{aligned}
& c\ne0, \quad c R +d R = R, \quad
\vert N(c)\vert \le \frac{2k}{k_{0}^{2}}, \\
& \vert \frac{d}{c} \vert < \epsilon_0\sqrt{\frac{2k}{N(c)^2k_{0}^{2}}-\frac{k_{0}^{2}}{2k}}+\frac{1+\omega}{2} \quad \text{and} \quad
\vert \frac{d'}{c'} \vert < \epsilon_0\sqrt{\frac{2k}{N(c)^2k_{0}^{2}}-\frac{k_{0}^{2}}{2k}}+\frac{1- \omega'}{2} .
\end{aligned}
\end{equation}
Then $\S_1$ is finite and if $R$ is a principal ideal domain then $\F=\F_{\infty}\cap \bigcap_{(c,d)\in \S_1} V^{\ge}_{c,d}$. In particular, $V_{c,d} \cap \F \neq \emptyset$ if and only if $(c,d) \in \S_1$.
\end{theorem}

\begin{proof}
That $\S_1$ is finite is a consequence of the well known fact that the set of elements of $R$ with a given norm is finite modulo multiplication by units and that the image of $R$ by the map $x\mapsto (x,x')$ is a discrete additive subgroup of $\R^2$, so that it intersects every compact subset in finitely many elements.

Assume now that $R$ is a PID. Then, by Lemma~\ref{minh}, 
$h_0(s_1,s_2,r)\ge \frac{k_0^2}{2k}>0$.
Therefore the set
	$$\S_{s_1,s_2,r} = \left\{(c,d) \in \S : c\ne 0 \text{ and } h_1(c,d,s_1,s_2,r) \ge \frac{k_0^2}{2k}\right\}$$
is not empty and
	\begin{equation}\label{h0Cota}
	 h_0(s_1,s_2,r)=\sup \{h_1(c,d,s_1,s_2,r) : (c,d)\in \S_{s_1,s_2,r}\}.
	\end{equation}

We claim that if $(s_1,s_2,r)\in B$ and $(c,d)\in \S_{s_1,s_2,r}$ then $(c,d)$ satisfies the conditions of (\ref{Conditionscd}).
Indeed, clearly $(c,d)$ satisfies the first two conditions.
It satisfies the third condition since
	\begin{equation*}
	 N(c)^2\left(\frac{k_0^2}{2k}\right)^2 \le N(c)^2h_1^2 \le f_{c,d,x_1,x_2,r}(h_1)=1.
	\end{equation*}
To prove that it satisfies the last two conditions of (\ref{Conditionscd}) let $x_1=s_1+s_2\omega$, $x_2=s_1+s_2\omega'$ and $h_1=h_1(c,d,s_1,s_2,r)$ Recall that $\vert s_1\vert ,\vert s_2 \vert \le \frac{1}{2}$ and hence
	\begin{equation}\label{Restrictionsx1x2}
	 \frac{-1-\omega}{2} \leq x_{1} \leq \frac{1+\omega}{2} \quad \text{and} \quad \frac{-1+\omega'}{2} \leq x_{2} \leq \frac{1-\omega'}{2}.
	\end{equation}
Furthermore
\begin{eqnarray*}
N(c)^{2}\frac{k_0^2}{2k}\left(\frac{k_0^2}{2k}+ \left(x_1+\frac{d}{c}\right)^2 r+\left(x_2+\frac{d'}{c'}\right)^2\frac{1}{r}\right)&=& \\
N(c)^{2}\left(\frac{k_0^2}{2k}\right)^2+ \left((cx_1+d)^2{c'}^2r+(c'x_2+d')^2c^2\frac{1}{r}\right)\frac{k_0^2}{2k}&\le& \\
N(c)^{2}h_{1}^{2}+ \left((cx_1+d)^2{c'}^2r+(c'x_2+d')^2c^2\frac{1}{r}\right)h_{1}+(cx_1+d)^2(c'x_2+d')^2&=& f_{c,d,x_1,x_2,r}(h_1)=1,
\end{eqnarray*}
This, together with $\epsilon_0^{-2}\le r \le \epsilon_0^2$ implies
\begin{equation*}\begin{aligned}
\begin{cases}
\vert x_{1}+\frac{d}{c} \vert < \frac{1}{\sqrt{r}}\sqrt{\frac{2k}{N(c)^2k_{0}^{2}}-\frac{k_{0}^{2}}{2k}} \le
	\epsilon_0\sqrt{\frac{1}{N(c)^2k_{1}}-k_{1}},\\
\vert x_{2}+\frac{d'}{c'} \vert < \sqrt{r}\sqrt{\frac{2k}{N(c)^2k_{0}^{2}}-\frac{k_{0}^{2}}{2k}} \le \epsilon_0\sqrt{\frac{1}{N(c)^2k_{1}}-k_{1}}.
\end{cases}\end{aligned}\end{equation*}
Combining this with (\ref{Restrictionsx1x2}) we also have
\begin{equation*}\begin{aligned}
\begin{cases}
\vert \frac{d}{c} \vert < \epsilon_0\sqrt{\frac{2k}{N(c)^2k_{0}^{2}}-\frac{k_{0}^{2}}{2k}}+\frac{1+\omega}{2},\\
\vert \frac{d'}{c'} \vert <\epsilon_0\sqrt{\frac{2k}{N(c)^2k_{0}^{2}}-\frac{k_{0}^{2}}{2k}}+\frac{1- \omega'}{2}.
\end{cases}\end{aligned}\end{equation*}
This proves the claim.

Combining the claim with (\ref{independent-unit2}) and (\ref{h0Cota}) we deduce that if $(s_1,s_2,r)\in B$ then
$h_0(s_1,s_2,r)=\sup \{h_1(c,d,s_1,s_2,r) : (c,d)\in \S_1 \cap \S_{s_1,s_2,r}\}$. As $\S_1$ is finite, this implies that $h_0(s_1,s_2,r)=h_1(c,d,s_1,s_2,r)$ for some $(c_0,d_0)\in \S_1$.
Then
$(s_1,s_2,r,h)\in \F$ if and only if $(s_1,s_2,r,h)\in V^{\ge}_{c_0,d_0}$ if and only if $(s_1,s_2,r,h)\in V^{\ge}_{c,d}$ for every $(c,d)\in \S$. Therefore
	$\F = \F_{\infty} \cap \bigcap_{(c,d)\in \S_1} V^{\ge}_{c,d}$
as desired.

To get the second part of the theorem, notice that $\F \cap V_{c,d} \subseteq \partial \F_0$ for $(c,d) \in \S$, $c \neq 0$. Moreover, similarly as in (\ref{fzero}),
\begin{equation*}
\partial \mathcal{F}_0 = \{(s_1,s_2,r,h)\in \R^2\times (\R^+)^2 : 
	h= h_0(s_1,s_2,r)\}.
\end{equation*}
Thus by the same reasoning as above, $\F_0 \cap V_{c,d} \neq \emptyset$ if and only if $(c,d) \in \S_1$ and hence the result follows.
\end{proof}


Recall that a collection of subsets of a topological space $X$ is said to be locally finite if every point
of $X$ has a neighbourhood intersecting only finitely many elements of the collection. As $\Hy_{r,s}$ is
locally compact, a collection of subsets of $\Hy_{r,s}$ is locally finite if every compact subset of
$\Hy_{r,s}$ intersects only finitely many elements of the collection.
A fundamental domain $F$ of a group $\h$ acting on $\Hy_{r,s}$ is said to be locally finite if
$\{g(F):g\in \h\}$ is locally finite.

\begin{lemma}\label{locallyfinite}
The fundamental domains $\mathcal{F}$ of $\Gamma$ and $\mathcal{F}_{\infty}$ of $\Gamma_{\infty}$ are
locally finite.
\end{lemma}

\begin{proof}
Let $C$ be a compact subset of $\HTwo$ and let $\gamma \in \Gamma_{\infty}$ such that $C \cap \gamma(\F) \neq \emptyset$. Let $Z=(x_1,x_2,y_1,y_2) \in C \cap \gamma(\F_{\infty})$. As $C$ is compact, the coordinates of $Z$ are bounded. As $\gamma^{-1} \in \Gamma_{\infty}$, $\gamma^{-1}=\begin{pmatrix} \epsilon_0^m & b \\ 0 & \epsilon_0^{-m} \end{pmatrix}$ for some $m \in \Z$ and $b \in R$, and hence $\gamma^{-1}(Z)=(\epsilon^{2m}x_1+b, \epsilon^{-2m}x_2+b', \epsilon^{2m}y_1, \epsilon^{-2m}y_2)$. As $y_1$ and $y_2$ are bounded, there are only finitely many $m \in \Z$ such that $\gamma^{-1}(Z) \in \F_{\infty}$, or equivalently $Z \in \gamma^{-1}(\F_{\infty})$, and this for every $Z \in C$. Moreover as the first two coordinates of $Z$ are bounded and $\{(b,b'):b\in R\}$ is a discrete subset of $R^2$ for each $m$ only finitely many $b$'s in $R$ satisfy that $\epsilon^{2m}x_1+b$ and $\epsilon^{-2m}x_2+b'$ satisfy the conditions imposed on $x_1$ and $x_2$ for $Z$ to be in $\F_{\infty}$. Thus there are only finitely many $\gamma \in \Gamma_{\infty}$ such that $C \cap \gamma(\F_{\infty}) \neq \emptyset$ and therefore $\F_{\infty}$ is locally finite.

Now suppose $C \cap \gamma(\F) \neq \emptyset$ for $C$ a compact subset of $\HTwo$ and $\gamma \in \Gamma$. If $\gamma \in \Gamma_{\infty}$, then we are done by the first part. So we may suppose that $\gamma=\begin{pmatrix} a & b \\ c & d \end{pmatrix}$ with $c \neq 0$. If $Z\in C\cap \gamma(\mathcal{F})$ and $c\ne 0$ then, by (\ref{equationfraction}), $\frac{1}{\Vert -cZ+a \Vert}=\Vert c\gamma^{-1}(Z)+d \Vert\ge 1$ and therefore $\Vert
-cZ+a \Vert\le 1$, in other words $(-c,a)\in V_{C,1}$, where $V_{C,1}$ is defined as in Lemma~\ref{BoundedFinite}. Using Lemma~\ref{BoundedFinite} we deduce that $(-c,a)$ belongs to a finite subset, up to units in $R$. So suppose $\gamma=\begin{pmatrix} a & b \\ c & d \end{pmatrix}$ and $\gamma_{u}= \begin{pmatrix} ua & u^{-1}b \\ u c & u^{-1}d \end{pmatrix}$ are such that $C \cap \gamma(\F) \neq \emptyset$ and $C \cap \gamma_u(\F) \neq \emptyset$ respectively for some $u \in \mathcal{U}(R)$. Then
$$ \gamma_u^{-1}\gamma=\begin{pmatrix} u^{-1} & * \\ 0 & u \end{pmatrix},$$
where $*$ denotes some element of $R$. Denote the latter matrix by $U$. Then $\gamma=\gamma_uU$ and $U \in \Gamma_{\infty}$. Hence we also have $C \cap \gamma_uU(\F) \neq \emptyset$ or equivalently $\gamma_u^{-1}(C) \cap U(\F) \neq \emptyset$. As $\gamma_u^{-1}(C)$ is still a compact subset of $\HTwo$ and $U(\F) \subseteq U(\F_{\infty})$, by the first part of the proof there are only finitely many units $u$ such that $\gamma_u$ satisfies $C \cap \gamma_u(\F) \neq \emptyset$ for every fixed $\gamma=\begin{pmatrix} a & b \\ c & d \end{pmatrix}$. For every $(-c,a)\in V_{C,1}$, fix a matrix $\gamma_{-c,a}=\begin{pmatrix}a & b_{-c,a} \\ c & d_{-c,a} \end{pmatrix}\in \Gamma$. Now consider an arbitrary matrix $\gamma=\begin{pmatrix} a & b \\ c & d \end{pmatrix} \in \Gamma$ and set $h=\gamma_{-c,a}\inv \gamma$. Then $h \in \Gamma_{\infty}$ and $h^{-1}(\gamma_{-c,a}^{-1}(C)) \cap \F_{\infty} \neq \emptyset$ if and only if $(\gamma_{-c,a}^{-1}(C) \cap h(\F_{\infty}) \neq \emptyset$. As $\gamma_{-c,a}^{-1}(C)$ is compact, by the first part of the proof there are only finitely many $h \in \Gamma_{\infty}$ that satisfy this and hence there are also only finitely many $\gamma \in \Gamma$ satisfying $C \cap \gamma(\F) \neq \emptyset$.
\end{proof}

\begin{corollary}\label{boule}
For every $Z \in \HTwo$, there exists $\lambda >0$ such that if $B(Z,\lambda) \cap g(\F) \neq \emptyset$, then $Z \in g(\F)$ for $g \in \h$.
\end{corollary}

\begin{proof}
Let $Z \in \HTwo$. Take some $\lambda' >0$ randomly. As the closed ball $\overline{B(Z,\lambda')}$ is compact, $\overline{B(Z,\lambda')} \cap g(\F) \neq \emptyset$ for only finitely many $g \in \h$. Hence also $B(Z,\lambda') \cap g(\F) \neq \emptyset$ for only finitely many $g \in \h$. Set $\lambda_0= \min \lbrace d(Z,g(\F)) \ : \ g \in \h \textrm{ such that } B(Z,\lambda') \cap g(\F) \neq \emptyset \textrm{ and } Z \not \in g(\F) \rbrace$, where $d(Z,g(\F))$ denotes the Euclidean distance from $Z$ to the the set $g(\F)$. Take $\lambda=\frac{\lambda_0}{2}$ and the corollary is proven.
\end{proof}

The following corollary is now easy to prove.

\begin{corollary}
$\partial \F = \bigcup_{g \neq 1} \F \cap g(\F)$.
\end{corollary}



\section{Some topological and geometrical properties of the fundamental domain}\label{topo}

In order to determine a presentation of $\Gamma$, we need to obtain more information on the fundamental domain $\F$ of $\Gamma$ constructed in the previous section. Thus in this section we will study some properties of $\F$. First we recall some notions from real algebraic geometry (for more details see \cite{BCR}).
Recall that a real algebraic set is the set of zeros in $\R^n$ of some subset of $\R\left[X_1,\ldots , X_n \right]$. An algebraic variety is an irreducible algebraic set, i.e. one which is not the union of two proper real algebraic subsets. Moreover, a real semi-algebraic set is a set of the form $\bigcup_{i=1}^s \bigcap_{j=1}^{n_i} \left\{ x \in \R^n :\ f_{i,j}(x) *_{i,j} 0 \right\}$, where $f_{i,j}(x) \in \R\left[X_1,\ldots , X_n \right]$ and $*_{i,j}$ is either $=$ or $<$, for $i=1, \ldots , s$ and $j=1, \ldots, n_i$. We will use the notion of dimension and local dimension of semi-algebraic sets as given in \cite{BCR}.
For example the sets $V_i^{\pm,\ge}$ and $V_{c,d}^{\ge}$ are real semi-algebraic sets and the sets $V_i^{\pm}$ and $V_{c,d}$ are real algebraic sets and in fact, we will prove that they are real algebraic varieties.

The following lemma describes when two sets $V_{c_1,d_1}$ and $V_{c_2,d_2}$ are equal.

\begin{lemma}\label{intermani}
Let $(c_1,d_1), (c_2,d_2) \in \S$. Then $V_{c_1,d_1} = V_{c_2,d_3}$ if and only if $(c_2,d_2)=(u c_1, u d_1)$ for some $u \in \U(R)$. Moreover if $c_1d_2=c_2d_1$ and $N(c_1)^2\ne N(c_2)^2$ then $V_{c_1,d_1}\cap V_{c_2,d_2}=\emptyset$.
\end{lemma}

\begin{proof}
One implication has already been given in (\ref{independent-unit5}). For the other implication, suppose that $V_{c_1,d_1}=V_{c_2,d_2}$. The case $c_1c_2=0$ follows from the fact that if $(c,d)\in \S$ then $c=0$ if and only if $V_{c,d}=\HTwo$. So assume $c_1c_2\ne 0$.
We can then rewrite $f_{c,d}$ as
\begin{equation*}
f_{c,d}(x_1,x_2,y_1,y_2)=
N(c)^2\left[\left(x_1+\frac{d}{c}\right)^2+y_1^2 \right]
\left[\left(x_2+\frac{d'}{c'}\right)^2+y_2^2 \right].
\end{equation*}
In particular, if $c\ne 0$ then the intersection of $V_{c,d}$ with $A=\{(x,0,y,1):x\in \R, y\in \R^+\}$ is formed by the points $(x,0,y,1)$ such that $(x,y)$ belongs to the ball with centre $\left(\frac{d}{c},0\right)$ and radius $\frac{1}{N(c)}\sqrt{1+\left(\frac{d'}{c'}\right)^2}$. Therefore, if $V_{c_1,d_1}=V_{c_2,d_2}$ then $\frac{d_1}{c_1}=\frac{d_2}{c_2}$, or equivalently, $c_1d_2=c_2d_1$, and $N(c_1)=N(c_2)$.
As $(c_1,d_1), (c_2,d_2) \in \S$, $g_1 = \begin{pmatrix} a_1 & b_1 \\ c_1 & d_1 \end{pmatrix} \in \h$ and $g_2=\begin{pmatrix} a_2 & b_2 \\ c_2 & d_2 \end{pmatrix} \in \h$ and $g_2g_1^{-1} \in \h_{\infty}$. Thus $g_2= \begin{pmatrix} u^{-1} & b \\ 0 & u \end{pmatrix} g_1$ for some $u \in \U(R)$ and $b \in R$ and so $(c_2,d_2)=(u c_1, u d_1)$.

The second part follows easily.
\end{proof}

Let
\begin{eqnarray*}
\V_{\infty} & = & \lbrace V_i^{+}, V_i^{-} :\  1 \leq i \leq 3 \rbrace, \\
\V & = & \left\{V_{c,d} :\  (c,d) \in \S, c \neq 0 \right\} \text{ and } \\
 \M &=& \left\{ \gamma(M) :  \gamma \in \Gamma \textrm{ and } M \in \V \cup \V_{\infty} \right\}.
\end{eqnarray*}

Observe that
\begin{equation}\label{Borders}
\partial \F_0 \subseteq \bigcup_{V \in \mathcal{V}} V, \quad  \partial \F_{\infty} \subseteq \bigcup_{V \in \mathcal{V}_{\infty}} V \quad \text{and} \quad \partial \F \subseteq \bigcup_{V\in \V \cup \V_{\infty}} V.
\end{equation}
Let $V\in \V\cup \V^{\infty}$. If $V=V_i^{\pm}$ (respectively, $V=V_{c,d}$) then let $V^{\ge}=V_i^{\pm,\ge}$ (respectively, $V^{\ge}= V_{c,d}^{\ge}$).
Define $V^{\le}$ in a similar way. Clearly, each $V^{\ge}$ and $V^{\ge}$ are real semi-algebraic sets, $V$ is the intersection with $\HTwo$ of a real algebraic set and it is the boundary of $V^{\le}$ and $V^{\ge}$ and its intersection.
Thus $\F_{\infty}, \F_0$ and $\F$ and their boundaries are real semi-algebraic sets.

Let $\gamma=\pmatriz{{cc} a&b\\c&d} \in \Gamma\setminus \{1\}$ such that if $c=0$ then $a>0$ and if moreover $a=1$ then $b=s_1+s_2\omega$ with $s_1,s_2\in \Z$. Then let
\begin{eqnarray}
E_{\gamma} = \left\{ \matriz{{ll}
			V_{c,d}^{\ge}, & \text{if } c\ne 0; \\
			V_3^{+,\le}, & \text{if } c=0 \text{ and } a<1; \\
			V_3^{-,\ge}, & \text{if } c=0 \text{ and } a>1; \\
			V_1^{+,\le}, & \text{if } c=0, a=1 \text{ and } b=s_1+s_2\omega \text{ with } s_1<0; \\
			V_1^{-,\ge}, & \text{if } c=0, a=1 \text{ and } b=s_1+s_2\omega \text{ with } s_1>0; \\
			V_2^{+,\le}, & \text{if } c=0, a=1 \text{ and } b=s_2\omega \text{ with } s_2<0; \\
			V_2^{-,\le}, & \text{if } c=0, a=1 \text{ and } b=s_2\omega \text{ with } s_2>0.}\right. \label{Egamma}
			\end{eqnarray}
Let $E'_{\gamma}$ be the set obtained by interchanging in the definition of $E_{\gamma}$ the roles of $\le$ and $\ge$ and let
	$$V_\gamma=\partial E_\gamma = E_\gamma\cap E'_\gamma.$$
The following lemma follows by straightforward calculations.

\begin{lemma}\label{ImageDFAlt}
 If $\gamma\in \Gamma\setminus \{1\}$ then $V_{\gamma}\in \V_{\infty}\cup \V$, $\F\subseteq E_\gamma$, $\gamma\inv(\F) \subseteq E'_{\gamma}$ and hence $\F\cap \gamma\inv(\F)\subseteq V_{\gamma}$.
\end{lemma}

Using (\ref{independent-unit5}) it is easy to prove the following lemma.

\begin{lemma}\label{Icdlocallyfinite}
The set $\V \cup \V_{\infty}$ is locally finite.
\end{lemma}

We first calculate the local dimensions of $\F$ and $\partial \F$.

\begin{lemma}\label{Floc4boundloc3}
The fundamental domain $\F$ has local dimension $4$ at every point and its boundary $\partial \F$ has local dimension $3$ at every point.
\end{lemma}

\begin{proof}
Let $Z \in \F$. We have to prove that for every $\lambda > 0$, $B(Z,\lambda) \cap \F$ is of dimension $4$.
This is obvious for $Z \in \F^{\circ}$.
So suppose $Z \in \partial \F$. If $Z \not \in \partial \F_0$, then $Z \in V$ for some $V \in \V_{\infty}$.
As $\partial \F_0$ is closed, there is $\lambda_0>0$ such that $B(Z,\lambda_0)\cap \partial F_0=\emptyset$. We may assume without loss of generality that $\lambda_0$ is smaller than 1 and $2\epsilon_0$.
If $Z=(s_1,s_2,r=s_3,h)$ then let $Z'=(s_1',s_2',r'=s_3',h)$ where for $i=1,2,3$ we have
	$$s_i'=\left\{ \matriz{{ll}
	s_i-\frac{\lambda_0}{\sqrt{3}}, & \text{if } Z \in V_i^+; \\
	s_i+\frac{\lambda_0}{\sqrt{3}}, & \text{if } Z \in V_i^-; \\
	s_i, & \text{otherwise}.}\right.$$
Then $Z' \in B(Z, \lambda_0) \cap \F_{\infty}^{\circ} \cap \F_0^{\circ}$ and thus there exists $\lambda' >0$ such that $B(Z',\lambda') \subseteq \F^{\circ} \cap B(Z,\lambda_0)$. Hence $B(Z,\lambda_0) \cap \F$ is of dimension $4$. By the choice of $\lambda_0$, this is true for any $\lambda < \lambda_0$ and of course also for any $\lambda > \lambda_0$. Finally suppose $Z \in \partial \F \cap \partial \F_0$. By (\ref{fzero}) and (\ref{fnotzero}) $Z$ is of the form $(s_1,s_2,r,h_0(s_1,s_2,r))$ with $(s_1,s_2,r) \in B$.
Let $\lambda >0$. By (\ref{fzero}) the point $Z'= (s_1,s_2,r,h_0(s_1,s_2,r)+\frac{\lambda}{2})$ is in $B(Z,\lambda) \cap \F_0^{\circ}$. If $Z' \in \F^{\circ}$, then there exists $\lambda' >0$ such that $B(Z',\lambda') \subseteq B(Z,\lambda) \cap \F^{\circ}$ and hence $B(Z,\lambda) \cap \F$ has dimension $4$. If not, then $Z' \in \partial \F \cap \F_0^{\circ}$ and hence by the above there exists $\lambda' >0$ such that $B(Z',\lambda') \subseteq B(Z,\lambda)$ and $B(Z',\lambda') \cap \F$ has dimension $4$. Thus also $B(Z,\lambda) \cap \F$ has dimension $4$.

To prove the second part, take $Z \in \partial \F$, $\lambda >0$ and set $\overline{B}=\overline{B(Z,\lambda)}$, $U_1=\overline{B} \cap \F$ and $U_2=\overline{B} \cap (\F^{c} \cup \partial \F)$, where $\F^{c}$ denotes the complementary of $\F$ in $\HTwo$. Then $\overline{B}$, $U_1$ and $U_2$ satisfy the conditions of Lemma~\ref{connected} and hence $B(Z,\lambda) \cap \partial \F = U_1 \cap U_2$ has dimension $3$.
\end{proof}

The next three lemmas give more details on the elements of $\V \cup \V_{\infty}$.

\begin{lemma}\label{sidemani1}
The elements of $\V \cup \V_{\infty}$ are non-singular irreducible real algebraic varieties of dimension $3$. Moreover if two different varieties $M_1$ and $M_2$ intersect non-trivially, with $M_1, M_2 \in \V \cup \V_{\infty}$, then  their intersection has local dimension $2$ at every point.
\end{lemma}

\begin{proof}
Applying the general implicit function theorem to the following functions
\begin{eqnarray*}
f_{c,d}(x_1,x_2,y_1,y_2) & = & \left[(cx_1+d)^2+c^2y_1^2 \right] \left[ ({c}' x_2+d')^2+{c}'^2y_2^2 \right],\\
F(x_1,x_2,y_1,y_2) & = & (f_{c_1,d_1}(x_1,x_2,y_1,y_2),f_{c_2,d_2}(x_1,x_2,y_1,y_2)),
\end{eqnarray*}
for $(c,d),(c_1,d_1),(c_2,d_2)\in \S$ with $c,c_1,c_2\ne 0$ and $V_{c_1,d_1}\ne V_{c_2,d_2}$, one gets that $V_{c,d}$  and $V_{c_1,d_1}\cap V_{c_2,d_2}$ are $\mathcal{C}^{\infty}$-manifolds of dimension 3 and 2 respectively. The fact that the elements of $\V$ are path-connected together with \cite[Proposition 3.3.10]{BCR} yield the desired property.
The same argument works for the elements of $\V_{\infty}$ or the combination of an element of $\V$ and an element of $\V_{\infty}$.
\end{proof}

\begin{lemma}\label{sidemani2}
The elements of $\M$ are non-singular irreducible real algebraic varieties of dimension $3$ and the intersection of two different elements of $\M$ is of dimension at most $2$.
\end{lemma}

\begin{proof}
The first part follows easily from Lemma \ref{sidemani1} and \cite[Theorem 2.8.8]{BCR} and the fact that the action of $\PSL_2(R)$ on $\HTwo$ is a bijective semi-algebraic map. The second part follows trivially from the irreducibility.
\end{proof}

\begin{lemma}\label{3inter}
Let $M_1,M_2,M_3$ be pairwise different elements of $\V \cup \V_{\infty}$ with $M_1\cap M_2\cap M_3 \ne \emptyset$. Then $M_1 \cap M_2 \cap M_3$ is a real algebraic set of local dimension $1$ at every point and with at most one singular point.
\end{lemma}

\begin{proof}
Suppose that $M_i=V_{c_i,d_i}$ for $i=1,2,3$. Similar, as in the proof of Lemma~\ref{sidemani1}, we apply the implicit function theorem to the function $F:\R^4 \rightarrow \R^3$ defined by
$$(x_1,x_2,y_1,y_2) \mapsto (f_{c_1,d_1}(x_1,x_2,y_1,y_2),f_{c_2,d_2}(x_1,x_2,y_1,y_2),f_{c_3,d_3}(x_1,x_2,y_1,y_2)).$$
So $V_{c_1,d_1} \cap V_{c_2,d_2} \cap V_{c_3,d_3}= F^{-1}(1,1,1)$. It is sufficient to prove that the Jacobian matrix of $F$ is of rank $3$ except possibly for at most one point. The Jacobian matrix of $F$ has the following form
\begin{eqnarray}\label{Jacobean}
2 \begin{pmatrix}
c_1(c_1x_1+d_1)\alpha_1 & c_2(c_2x_1+d_2)\alpha_2 & c_3(c_3x_1+d_3)\alpha_3 \\
c_1'(c_1'x_2+d_1')\beta_1 &
c_2'(c_2'x_2+d_2')\beta_2 & c_3'(c_3'x_2+d_3')\beta_3 \\
c_1^2y_1\alpha_1 & c_2^2y_1\alpha_2 & c_3^2y_1\alpha_3 \\
c_1'^2y_{2}\beta_1 & c_2'^2y_{2}\beta_2 & c_3'^2y_{2}\beta_3
\end{pmatrix}, \end{eqnarray}
where $\alpha_i=(c_i'x_2 + d_i')^2+c_i'^2y_2^2$ and $\beta_i=(c_ix_1+d_i)^2+c_i^2y_1^2$. Note that $\beta_i=\alpha_i^{-1}$ because $(x_1,x_2,y_1,y_2) \in V_{c_1,d_1} \cap V_{c_2,d_2} \cap V_{c_3,d_3}$. In particular $\alpha_i$ is nonzero for every $i=1,\ldots ,3$. The determinant $det_1$ of the submatrix of (\ref{Jacobean}) formed by the first three rows is
\begin{eqnarray*}
det_1 & = & (\gamma_1(c_1'x_2+d_1') + \gamma_2(c_2'x_2+d_2') + \gamma_3(c_3'x_2+d_3'))y_1,
\end{eqnarray*}
with $\gamma_i=c_i'c_{i+1}c_{i+2}\alpha_i^{-1}\alpha_{i+1}\alpha_{i+2}(c_{i+1}d_{i+2}-c_{i+2}d_{i+1})$, where the indexes are to be interpreted modulo $3$. Note that each $\gamma_i \neq 0$ because of Corollary~\ref{intermani} and the assumption that the three varieties considered are distinct. Similar calculations show that the determinant $det_2$ of the submatrix of (\ref{Jacobean}) formed by the rows $1$, $3$ and $4$ is:
\begin{eqnarray*}
det_2=(-\gamma_1c_1'-\gamma_2c_2'-\gamma_3c_3')y_1y_2.
\end{eqnarray*}
If both determinants are $0$, then (because $y_1 \neq 0$ and $y_2 \neq 0$)
\begin{equation*}\begin{aligned}
\begin{cases}
\gamma_1c_1'+\gamma_2c_2'+\gamma_3c_3' = 0 \\
\gamma_1d_1'+\gamma_2d_2'+\gamma_3d_3' = 0.
\end{cases}\end{aligned}\end{equation*}
Hence the vector $(\gamma_1,\gamma_2,\gamma_3)$ is perpendicular to the vectors $(c_1',c_2',c_3')$ and $(d_1',d_2',d_3')$ and thus
\begin{equation*}
(\gamma_1,\gamma_2,\gamma_3) = t ((c_1',c_2',c_3') \times (d_1',d_2',d_3'))
\end{equation*}
for some $t \in \R$. So
\begin{equation*}
\begin{pmatrix}
c_1'c_{2}c_{3}\alpha_1^{-1}\alpha_{2}\alpha_{3}(c_{2}d_{3}-c_{3}d_{3}) \\
c_2'c_{3}c_{1}\alpha_2^{-1}\alpha_{3}\alpha_{1}(c_{3}d_{1}-c_{1}d_{3})\\
c_3'c_{1}c_{2}\alpha_3^{-1}\alpha_{1}\alpha_{2}(c_{1}d_{2}-c_{2}d_{1})
\end{pmatrix}
= t \begin{pmatrix}
c_{2}'d_{3}'-c_{3}'d_{2}'\\
c_{3}'d_{1}'-c_{1}'d_{3}'\\
c_{1}'d_{2}'-c_{2}'d_{1}'
\end{pmatrix}.
\end{equation*}
Dividing the first coordinate by the second and the third, we get that
\begin{eqnarray}
\alpha_1^{-2} & = & \mu_2 \alpha_2^{-2} \label{propo1} \\
\alpha_1^{-2} & = & \mu_3 \alpha_3^{-2}, \label{propo2}
\end{eqnarray}
for some $\mu_2$ and $\mu_3$ depending only on $c_i$ and $d_i$ for $1 \leq i \leq 3$.
Because $\alpha_i^{-1}$ can be interpreted as $c_i^2$ times the square of the Euclidean distance from the point $(x_1,y_1)$ to the point $(-\frac{d_i}{c_i}, 0)$, equations (\ref{propo1}) and (\ref{propo2}) take the form
	$$d\left((x_1,y_1),\left(-\frac{d_1}{c_1},0\right)\right)=\lambda_i d\left((x_1,y_1),\left(-\frac{d_i}{c_i},0\right)\right)$$
with $\lambda_i=\frac{c_i}{c_1}\sqrt[4]{\mu_i}$ for $i=2,3$. Hence, modulo a translation $(x_1,y_1)$ is a solution for $(x,y)$ of the system
\begin{eqnarray*}\begin{aligned}
\begin{cases}
x^2+y^2=\lambda_2^2((x-a)^2+y^2)\\
x^2+y^2=\lambda_3^2((x-b)^2+y^2),
\end{cases}\end{aligned}\end{eqnarray*}
when $a$ and $b$ are different non-zero real numbers.
If $\lambda_i=1$ for $i=2$ or $i=3$, then the corresponding equation represents the bisector (which is a line) of $(0,0)$ and $(a,0)$ or of $(0,0)$ and $(b,0)$. Otherwise, the equation represents the circle $\mathcal{C}_i$
with centre $\left(-\frac{\lambda_{i}^2a}{1-\lambda_{i}^2},0\right)$  and radius
$\frac{\lambda_2a}{1-\lambda_2^2}$ (because $y >0$ and $a \neq b$, the case $\lambda_1=1=\lambda_2$ is impossible). Such two different lines, two different circles or a circle and a line
intersect in at most one point. So, if there are more than two points in $\Hy^2 \times \Hy^2$ satisfying the two equations then $\lambda_i\ne 1$ for $i\ne 2,3$ and the equations represent the same circles. In this case
\begin{eqnarray*}\begin{aligned}\begin{cases}
\frac{\lambda_2^2a}{1-\lambda_2^2} = \frac{\lambda_3^2b}{1-\lambda_3^2}\\
\frac{\lambda_2a}{1-\lambda_2^2} = \frac{\lambda_3b}{1-\lambda_3^2}.
\end{cases}\end{aligned}\end{eqnarray*}
Dividing the first equation by the second we conclude that $\lambda_2=\lambda_3$ thus $a=b$, a contradiction.

So we have shown that there exists at most one possible couple $(x_1,y_1)$ which can be completed to a point $(x_1,x_2,y_1,y_2)$ in the intersection and such that $det_1=det_2=0$.
Therefore, the condition that the rank of the Jacobian matrix in $(x_1,x_2,y_1,y_2)$ is less than 3 determines the coordinates $x_1$ and $y_1$.
By symmetry, it also determines the coordinates $x_2$ and $y_2$. This means that the intersection $V_{c_1,d_1} \cap V_{c_2,d_2} \cap V_{c_3,d_3}$ is a real algebraic variety of dimension at most $1$.

A similar argument shows that the result remains true if one of several of $M_1$, $M_2$ and $M_3$ are elements of $\V_{\infty}$.
\end{proof}

In order to determine the border of $\F$, we need the following definition.

\begin{definition}\label{definingmani}
An essential hypersurface of $\F$ is an element $M \in \V \cup \V_{\infty}$ such that $M \cap \F$ is of dimension $3$.
Let $\V_e$ denote the set of essential hypersurfaces of $\F$.
\end{definition}


Following the notation of \cite{BCR}, for $A$ a real semi-algebraic set, we denote by $A^{(d)}$ the set of points of $A$ with local dimension $d$, i.e.
\begin{eqnarray} \label{LocalDimensionEq}
 A^{(d)} &=& \lbrace Z \in A \ :\  \forall\  \lambda > 0,\  dim(B(Z,\lambda) \cap A) = d \rbrace.
 \end{eqnarray}

\begin{lemma}\label{borderdefmani}
$$\partial \F = \bigcup_{M\in \V_e} (M \cap \F)=\bigcup_{M\in \V_e} (M \cap \F)^{(3)},$$
\end{lemma}

\begin{proof}
Clearly $\bigcup_{M\in \V_e} (M \cap \F)^{(3)}\subseteq \bigcup_{M\in \V_e} (M \cap \F) \subseteq \partial \F$.
Let $Z \in \partial \F$ and $\V_Z=\left\{ M \in \V \cup \V_{\infty} : \ Z \in M \right\}$.
We have to prove that $M\cap F$ has local dimension 3 at $Z$ for some $M\in \V_Z$.
By (\ref{Borders}) and Lemma~\ref{Icdlocallyfinite}, $\V_Z$ is a non-empty finite set.
Thus there is an open ball $B=B(Z,\lambda_0)$ such that for every $M \in \V \cup \V_{\infty}$, $B\cap M \neq \emptyset$ if and only if $M\in \V_Z$.
Thus, if $0<\lambda \le \lambda_0$,  then $\partial \F \cap B(Z,\lambda) = \cup_{M\in \V_Z} \F \cap B \cap M$ and by Lemma~\ref{sidemani1}, $\partial \F \cap B(Z,\lambda)$ has dimension $3$.
Thus $\F \cap B(Z,\lambda) \cap M$ has dimension 3 for some $M\in \V_Z$. Therefore $M\cap \F$ has local dimension 3 at $Z$.
\end{proof}

\begin{lemma}\label{Vinfloc3}
If $V \in \V_{\infty}$, then $V\cap F$ has local dimension 3 at every point, i.e. $(V \cap \F)^{(3)} = V \cap \F$.
\end{lemma}

\begin{proof} This follows by arguments similar to those used in the proof of Lemma~\ref{Floc4boundloc3}.
\end{proof}

The following proposition gives information about essential hypersurfaces and will be important in the next two sections.

\begin{proposition}\label{esshyploc3}
Assume $Z \in \partial\F$ and $Z$ is contained in at most two elements $M$ and $M'$ of $\V \cup \V_{\infty}$. Then $M$ and $M'$ are essential and for every $\lambda >0$, the intersections $B(Z, \lambda) \cap M \cap \F$ and $B(Z, \lambda) \cap M' \cap \F$ are of dimension $3$.
\end{proposition}

\begin{proof}
If $Z$ is contained in a single element $M$ of $\V \cup \V_{\infty}$, then the statement is obvious.
So suppose that $Z \in M \cap M'$ for some $M,M' \in \V \cup \V_{\infty}$ with $M \neq M'$ and $Z$ is contained in no other element of $\V \cup \V_{\infty}$. By Lemma~\ref{sidemani1}, $M \cap M'$ is of dimension $2$. Moreover $M=(M \cap M'^{<}) \cup (M \cap M') \cup (M \cap M'^{>})$ and, as by the proof of Lemma~\ref{sidemani1}, both varieties $M$ and $M'$ are not tangent in $Z$ the first and the third
intersections are of dimension $3$. Thus $B(Z, \lambda) \cap M \cap M'^{>}$ is of dimension $3$ and is contained in the
boundary of $\F$. Hence $M$ is an essential hypersurface and $B(Z, \lambda) \cap M \cap \F$ is of dimension $3$.
Inverting the role of $M$ and $M'$, we get the same result for $M'$.
\end{proof}

Finally, if $R$ is a PID, we can prove compactness of the intersection of $\F$ with the elements of $\V$.

\begin{lemma}\label{esscompact}
Assume $R$ is PID and let $V_{c,d}$, for $(c,d) \in \S$ with $c \neq 0$. Then $V_{c,d} \cap \F$ is compact.
\end{lemma}

\begin{proof}
Let $Z=(s_1,s_2,r,h) \in V_{c,d}  \cap \F$. As $V_{c,d} \cap \F \subseteq \F_{\infty}$, $\vert s_1 \vert \leq \frac{1}{2}$, $\vert s_2 \vert \leq \frac{1}{2}$ and $\epsilon_0^{-2} \leq r \leq \epsilon_0^{2}$. Moreover $V_{c,d}  \cap \F \subseteq \F_0$ and hence, by Lemma~\ref{minh}, $h > \frac{k_0^2}{2k}$. As $Z \in \partial \F$, Lemma~\ref{secureheight} yields that $h \leq 1$. Thus $V_{c,d} \cap \F$ is hyperbolically bounded and closed in $\HTwo$ and hence it is compact.
\end{proof}

\section{Generators of $\h$}\label{sidepairing}

In the remainder of the paper, $\te$ denotes the tessellation of $\HTwo$ given by $\Gamma$ and $\F$, i.e $\te=\{\gamma(\F):\gamma\in \Gamma\}$.
In order to get generators for $\h$, we have to analyse the intersections between elements of $\te$.
The next lemma is crucial.

\begin{lemma}\label{defside2}
Let $T_1$, $T_2$ and $T_3$ be three different elements of $\te$. Then $T_1 \cap T_2 \cap T_3 \subseteq M_1 \cap M_2$ for $M_1, M_2$ two different elements of $\M$. In particular the intersection of three different tiles has dimension at most $2$.
\end{lemma}

\begin{proof}
For every $i=1,2,3$, let $\gamma_i\in \Gamma$ with $T_i=\gamma_i(\F)$.
Let $\mathcal{N}=\{M\in \M : T_i\cap T_j \subseteq M \text{ for some } 1\le i\le j \le 3\}$.
By Lemma~\ref{ImageDFAlt}, for every $1\le i<j\le 3$, we have $T_i\cap T_j \subseteq \gamma_i(V_{\gamma_j\inv \gamma_i})$.
Thus $\mathcal{N}\ne \emptyset$ and it is enough to show that $\mathcal{N}$ has at least two different elements.
By means of contradiction, assume that $\mathcal{N}=\{M=\gamma(V)\}$ for some $V\in \V_{\infty}\cap \V$ and $\gamma\in \Gamma$.
Let $M^{\ge}=\gamma(V^{\ge})$ and $M^{\le}=\gamma(V^{\le})$.
Then $M=\gamma_i\left(V_{\gamma_j\inv \gamma_i}\right)$ and, by Lemma~\ref{ImageDFAlt}, for each $1\le i<j \le 3$ either $T_i\subseteq M^{\ge}$ and $T_j\subseteq M^{\le}$ or viceversa.
By symmetry one may assume that $T_1\subseteq M^{\ge}$. Then $T_2\subseteq M^{\le}$ and hence $T_3\subseteq M^{\le}\cap M^{\ge}=M$, a contradiction because $\dim(T_3)=4$ and $\dim(M)=3$.
\end{proof}

\begin{definition}
For each $\gamma\in \h\setminus \{1\}$, set
	$$S_{\gamma}=\F \cap \gamma\inv(\F).$$
A \emph{side} of $\F$ with respect to $\Gamma$ is a set of the form $S_{\gamma}$ that has dimension 3.
\end{definition}
In that case, we say that $\gamma$ is a \emph{side pairing transformation} of $\Gamma$ with respect to $\F$, or simply a \emph{pairing transformation}.
Observe that if $\gamma$ is a pairing transformation then so is $\gamma\inv$ and $\gamma(S_{\gamma})=S_{\gamma\inv}$.
Hence the pairing transformations ``pair'' the sides of $\F$.
More generally, a \emph{side} of $\te$ is a side of $\gamma(\F)$ for some $\gamma\in \Gamma$.
Equivalently, the sides of $\te$ are the sets of dimension 3 of the form $\gamma(\F)\cap \phi(\F)$, with $\gamma,\phi\in \Gamma$.

We now can state the main result of this section.

\begin{theorem}\label{Poincare}
Let $\F$ be the fundamental domain of $\Gamma=\PSL_2(R)$ described in Theorem~\ref{funddomain}. Then $\Gamma$ is generated by the pairing transformations of $\F$ with respect to $\Gamma$.
\end{theorem}

\begin{proof}
Let $L= \lbrace \bigcap_{i \in I} \gamma_i(\F) :\ dim(\bigcap_{i \in I} \gamma_i(\F)) \leq 2 \rbrace$ and consider the set
$$\Omega = \HTwo \setminus \bigcup_{Y \in L} Y.$$
This is a set of elements $Z \in \HTwo$ that belong either to the interior of a tile of $\F$ or to a unique side of $\te$. Indeed, let $Z \in \Omega$ and suppose $Z$ is not contained in the interior of a tile of $\F$. Then $Z \in \gamma_1(\F) \cap \gamma_2(\F)$ for at least two distinct elements $\gamma_1 \neq \gamma_2 \in \h$. Because of Lemma~\ref{defside2}, $\gamma_1(\F)$ and $\gamma_2(\F)$ are the only tiles containing $Z$. Thus, $Z \in \gamma_1(\F) \cap \gamma_2(\F)$ and $Z \not \in \gamma(\F)$ for $\gamma \in \h \setminus \lbrace \gamma_1,\gamma_2 \rbrace$. By the definition of $\Omega$, $dim(\gamma_1(\F) \cap \gamma_2(\F))=3$ and hence $\gamma_1(\F) \cap \gamma_2(\F)$ is the unique side containing $Z$.

Note that $\X=\HTwo$ and $L$ satisfy the hypotheses of Lemma~\ref{pathconnected} and hence $\Omega$ is path-connected. Let $\gamma \in \h$ and $Z \in \F^{\circ}$. Put $W=\gamma(Z)\in \gamma(\F)^{\circ}$. There exists a path $p$ in $\Omega$ joining $Z$
and $W$.
Let $A = \lbrace h \in \Gamma\ :\ p \cap h(\F) \neq \emptyset \rbrace$. As $p$ is compact and
as $\F$ is locally finite by Lemma~\ref{locallyfinite}, $A$ is finite. We define recursively a sequence of subsets of $A$ by
setting $A_0=\lbrace 1 \rbrace$ and if $i \geq 1$ then $A_i= \lbrace h \in A \ :\ h(\F) \cap k(\F) \textrm{ is a side for some }
k\in  A_{i-1} \setminus \bigcup_{j<i-1} A_j \rbrace$. Let $B=\bigcup_{i \geq 0} A_i$. We claim that $B=A$. Otherwise let $\alpha:
\left[0,1\right] \rightarrow p \subseteq \Omega$ be a continuous function with image $p$ and such that $\alpha(0)=Z$ and $\alpha(1)=\gamma(Z)$ and let $a=min\lbrace t \in \left[0,1\right]\ :\ \alpha(t) \in h(\F) \textrm{ for some } h \in A \setminus B \rbrace$. This minimum exists because $A
\setminus B$ is non-empty and $\bigcup_{k \in A \setminus B} k(\F)$ is closed. Moreover $a >0$ because $\alpha(0)=Z \in
\F^{\circ}$ and $1 \in B$, so that $Z \not \in h(\F)$ for each $h \in A \setminus B$. Then $\alpha(\left[0,a\right)) \subseteq
\bigcup_{h \in B} h(\F)$ and as this union is closed, $\alpha(a) \in h(\F) \cap k(\F)$ for some $h \in B$ and $k \in A \setminus
B$. As $\alpha(a) \in \Omega$, $h(\F) \cap k(\F)$ has dimension $3$ and hence it is a side.
This contradicts the definition of $B$. Hence $A=B$ and in particular $\gamma \in B$.
By using the sets $A_i$, we create a sequence $\gamma_0=1,\gamma_1,\dots,\gamma_k=g$ such that for every $1 \leq j $, $\gamma_{j-1}(\F) \cap \gamma_{j}(\F)$ is a side.
Hence $\F \cap \gamma_{j-1}^{-1}\gamma_{j}(\F)$ is a side of $\F$ and thus $\gamma_{j-1}^{-1}\gamma_{j}$ is one of the proposed
generators. As $\gamma=(\gamma_0^{-1}\gamma_{i_1})(\gamma_{i_1}^{-1}\gamma_{i_2})\ldots(\gamma_{i_{k}-1}^{-1}\gamma_{i_k})$, the result follows.
\end{proof}

Recall that the border of $\F$ is covered by the essential hypersurfaces of $\F$. In fact it is also covered by the sides of $\F$.

\begin{lemma}\label{borderunionsides}
The border of $\F$ is the union of the sides of $\F$.
\end{lemma}

\begin{proof}
Since the sides of $\F$ are contained in $\partial \F$, we only have to prove that if $Z \in \partial \F$, then $Z$ belongs to a side of $\F$. As $Z$ is in $\partial \F$ and $\F$ is a locally finite a fundamental domain, $Z \in \gamma(\F)$ for only finitely many $\gamma \in \Gamma$, say $\gamma_1=1, \ldots ,  \gamma_k$. Choose $\lambda_0 >0$ such that the closed ball $\overline{B(Z,\lambda_0)}$ intersects $\gamma_i(\F)$ if and only if $Z \in \gamma_i(\F)$ (this is possible by Corollary~\ref{boule}). Then
$$Z \in \overline{B(Z,\lambda_0)} = \left( \overline{B(Z,\lambda_0)} \cap \F \right) \cup \left(  \overline{B(Z,\lambda_0)} \cap \cup_{i=2}^k \gamma_i(\F) \right).$$
Thus by Lemma \ref{connected}, $\F \cap \gamma_i(\F)$ is of dimension $3$, for at least one $1 \leq i \leq k$ and hence $Z$ is contained in a side of $\F$.
\end{proof}

In order to give more precise information on the generators of $\Gamma$ described in Theorem~\ref{Poincare}, we need to analyse the relationship between pairing transformations and essential hypersurfaces.
By Lemma~\ref{borderdefmani}, $\partial \F$ is the union of the sets of the form $(\F \cap V)^{(3)}$ with $V$ running through the essential hypersurfaces and by Lemma~\ref{borderunionsides}, $\partial \F$ also is the union of the sides of $\F$ with respect to $\Gamma$.
Moreover, if $\gamma$ is a pairing transformation then $S_{\gamma}\subset \F\cap V_{\gamma}$, by Lemma~\ref{ImageDFAlt}.
Hence $V_{\gamma}$ is an essential hypersurface of $\F$ and it is the unique essential hypersurface of $\F$ containing $S_{\gamma}$.
Conversely, let $V$ be an essential hypersurface and let
	$$\Gamma_V = \{\gamma : \gamma \text{ is a pairing transformations such that }S_{\gamma}\subseteq V\}.$$
Clearly, $\gamma\in \Gamma_V$ if and only if $\dim S_{\gamma}=3$ and $V=V_{\gamma}$. Moreover,
	$$(V\cap \F)^{(3)} \subseteq \bigcup_{\gamma\in \Gamma_V} S_{\gamma}.$$
Each pairing transformation belongs to $\Gamma_V$ for some essential hypersurface $V$. We will show that, in order to generate $\Gamma$ it is enough to take one element of $\Gamma_V$ for each essential hypersurface.

We start dealing with the elements of $\V_{\infty}$.
Clearly $\Gamma_{\infty}$ is generated by the following elements:
$$P_1= \begin{pmatrix} 1 & 1 \\ 0 & 1 \end{pmatrix}, \quad
P_2 = \begin{pmatrix} 1 & \omega \\ 0 & 1 \end{pmatrix},\quad
P_3 = \begin{pmatrix} \epsilon_0 & 0 \\ 0 & \epsilon_0^{-1} \end{pmatrix}.$$
Using (\ref{h0atmost1}) and (\ref{fnotzero}) it is easy to see that the six sets $V_i^{\pm}$, with $i=1,2,3$, are essential hypersurfaces.
Moreover, a straightforward calculation shows that if $i=1,2$ then
	\begin{equation}\label{GammaV12}
	 \Gamma_{V_i^{\pm}} = \{P_i^{\mp 1}\} \quad \text{(with opposite signs on both sides)}.
	\end{equation}
Thus, for $i=1,2$, $P_i$ (respectively, $P_i\inv$) is a pairing transformation and its side covers the part of the boundary given by $V_i^+ \cap \F$ (respectively, $V_i^-\cap F$).

However to cover $\F \cap V_3^{\pm}$ with sides, we may need more than one side.

\begin{lemma}\label{GammaV3}
If $g\in \Gamma_{V_3^{\pm}}$ then $g=QP_3^{\mp}$ (with opposite signs on both sides) for some $Q\in \GEN{P_1,P_2}$ and the inversion map is a bijection $\Gamma_{V_3^+}\rightarrow \Gamma_{V_3^-}$. Moreover $\Gamma_{V_3^+}$ and $\Gamma_{V_3^+}$ are finite.
\end{lemma}
\begin{proof}
The first statement follows from Lemma~\ref{ImageDFAlt}, Lemma~\ref{sidemani1}, (\ref{Egamma}) and some easy computations.
If $\gamma\in \Gamma_{V_3^+}$ then $\gamma=\pmatriz{{cc} \epsilon_0\inv & b \\ 0 & \epsilon_0}$ with $b\in R$. Moreover, by (\ref{chapeau1}) and (\ref{chapeau2}), $\F \cap \gamma(\F)$ is a non-empty subset formed by elements of the form $(x_1,x_2,y_1,y_2)$ such that $(x_1,x_2)$ and
$(\widehat{x_1}, \widehat{x_2})=(\epsilon_0^{-2}x_1+\epsilon_0^{-1}b, \epsilon_0^{2}x_2+\epsilon_0 b')$.
belong to a compact set. As $\left\{(b,b'):\ b \in R \right\}$ is discrete, we deduce that $b$ belongs to a finite subset of $R$.
Thus $\Gamma_{V_3^+}$ is finite. Hence $\Gamma_{V_3^-}$ is finite too.
\end{proof}

We now deal with the essential hypersurfaces of the form $V_{c,d}$ with $(c,d)\in \S$ (and necessarily $c\ne 0$).

\begin{lemma}\label{GammaVcd}
Let $(c,d)\in \S$ with $c\ne 0$.
Then the second row of every element of $\Gamma_{V_{c,d}}$ is of the form $(uc,ud)$ for some $u\in \U(R)$.
Equivalently, if the second row of $\gamma\in \Gamma$ is $(c,d)$ then $\Gamma_{V_{c,d}}\subseteq \gamma \Gamma_{\infty}$.
\end{lemma}

\begin{proof}
Let $\gamma\in \Gamma_{V_{c,d}}$ and let $v\in R^2$ be the second row of $\gamma$. Then $S_{\gamma}\subseteq V_{\gamma}\cap V_{c,d}$, by Lemma~\ref{ImageDFAlt}, and hence $V_{\gamma}=V_{c,d}$, by Lemma~\ref{sidemani2}. Then $v=(uc,ud)$ for some $u\in \U(R)$, by Lemma~\ref{intermani}.
\end{proof}

Let $\S_e$ denote the set of $(c,d)\in \S$ such that $V_{c,d}$ is an essential hypersurface of $\F$.

\begin{corollary}\label{GeneratorsGeneral}
For every $(c,d)\in \S_e$ choose $a,b\in R$ with $P_{c,d}=\pmatriz{{cc} a & b \\ c & d}\in \Gamma$.
Then $\Gamma=\GEN{P_1,P_2,P_3,P_{c,d} : (c,d)\in \S_e}$.
\end{corollary}

\begin{proof}
By Theorem~\ref{Poincare}, $\Gamma$ is generated by $\cup_{V\in \V_e} \Gamma_V$ and hence it is enough to show that $\GEN{P_1,P_2,P_3,\gamma_{c,d} : (c,d)\in \S}$ contains $\Gamma_V$ for each $V\in \V_e$. This is a consequence of (\ref{GammaV12}), Lemma~\ref{GammaV3}, Lemma~\ref{GammaVcd} and the fact that $\Gamma_{\infty}$ is generated by $P_1$, $P_2$ and $P_3$.
\end{proof}

In case $R$ is a PID we can combine Theorem~\ref{FiniteSides} and Corollary~\ref{GeneratorsGeneral} to get the following

\begin{corollary}\label{generators}
Suppose that $R$ is a PID and let $\S_1$ be as in Theorem~\ref{FiniteSides}.
For each $(c,d)\in \S_{1}$ choose $a,b\in R$ such that $P_{c,d}=\begin{pmatrix} a&b\\c&d \end{pmatrix} \in \h$. Then
 $\h =\langle P_{1}, \; P_{2},\; P_{3} , P_{c,d} \mid (c,d)\in \S_{1}\rangle$.
\end{corollary}


\section{Defining Relations of $\h$}\label{defrel}

In this section we construct the relations associated to the pairing transformations, and hence we obtain a presentation of the group $\h$.
As in the classical case, there are two types of relations. These will be called the \emph{pairing relations} and the \emph{cycle relations}. The pairing relations are quite obvious to establish.

\begin{notation}
Given a side $S$ of $\F$, let $\gamma_S$ denote the unique element of $\Gamma$ such that $S= \F \cap \gamma_S^{-1}(\F)$ and let $S^*=\F \cap \gamma_S(\F)$.
\end{notation}

Observe that if $S$ is a side then $\gamma_S(S)=S^*$.
Moreover $S\mapsto \gamma_S$ and $\gamma \mapsto S_{\gamma}$ define mutually inverse bijections between the sides of $\F$ and the pairing transformations such that $\gamma_S\inv=\gamma_{S^*}$, or equivalently $S_{\gamma}^*=S_{\gamma\inv}$.
The pairing relation given by $S$ is then simply $\gamma_S\gamma_{S^{*}}=1$. Note that in case  $S=S^{*}$, we get as pairing relation $\gamma_S^2=1$.

We will now turn to the cycle relations. For this we need to introduce the following definition.

\begin{definition}\label{defside}
A \emph{cell} $C$ of $\te$ is a non-empty intersection of tiles of $\F$ satisfying the following property: for every $\gamma \in \h$, either $C \subseteq \gamma(\F)$ or $dim(C \cap \gamma(\F) ) \leq dim(C) -1$.
Clearly, the cells of dimension 4 are the tiles. By Lemma~\ref{defside2}, the sides of $\te$ are the cells of dimension 3.
A cell of dimension $2$ is called an \emph{edge}.
If a cell or edge is contained in a tile $T$, then it is called a cell or an edge of $T$.
\end{definition}

Observe that a cell is always a finite intersection of tiles. Indeed, consider $Z \in C$. As $\lbrace Z \rbrace$ is compact, $Z$ is contained in only finitely many tiles of $\F$, and hence so is $C$. The following proposition generalizes in some sense the notion of relative interior of a cell.

\begin{proposition}\label{relintside}
Let $C$ be a cell of $\te$. Then there exists $Z \in C$, such that, for every $\gamma \in \Gamma$, $Z \in \gamma(\F)$ if and only if $C \subseteq \gamma(\F)$.
\end{proposition}

\begin{proof}
Let $\gamma_1, \ldots , \gamma_n$ be the elements $\gamma \in \Gamma$ such that $C \subseteq \gamma(\F)$. Suppose, by contradiction, that there do not exist $Z \in C$ that satisfy the statement of the proposition. Then, for every $Z \in C$, there exists $\gamma_Z \in (\Gamma \setminus  \lbrace \gamma_i :\ 1 \leq i \leq n \rbrace )$ such that $Z \in \gamma_Z(\F)$. Put $\h^{*} =  \lbrace \gamma_Z :\ Z \in C \rbrace$. This is a countable set because $\h$ is countable, and clearly $C \subseteq \cup_{\gamma \in \h^{*}} \gamma(\F)$. Thus
$$ C = \cup_{\gamma \in \h^{*}} \left( \gamma(\F) \cap C \right).$$
As $C \not \subseteq \gamma(\F)$ for $\gamma \in \h^{*}$ and because $C$ is a cell, $dim(C \cap \gamma(\F) ) \leq dim(C)-1$ and hence $C=\cup_{\gamma \in \h^{*}} \left( \gamma(\F) \cap C \right)$ has dimension at most $dim(C)-1$, a contradiction.
\end{proof}

The next lemma is an obvious consequence of the definition of an edge.

\begin{lemma}\label{edgeonedge}
Let $\gamma \in \h$. If $E$ is an edge then $\gamma(E)$ is an edge. In particular, if $E$ is an edge of $\F$ contained in some side $S_\gamma = \F \cap \gamma\inv(\F)$, with $\gamma \in \Gamma$, then $\gamma(E)$ is an edge of $\F$.
\end{lemma}

In order to prove more results on the edges of $\F$, we first have to analyse the sides a bit more in detail.

\begin{lemma}\label{imageh}
Let $(c,d) \in \S$ with $c \neq 0$. If $t \in \h_{\infty}$ then there exists $(c_0,d_0)\in \S$, $c_0 \neq 0$, such that $t(V_{c,d})=V_{c_0,d_0}$.
\end{lemma}

\begin{proof}
Let $Z \in V_{c,d}$. Fix $\alpha, \beta \in R$
such that $\alpha c+\beta d=1$. Write $t=\begin{pmatrix} \epsilon & b \\ 0 & \epsilon^{-1} \end{pmatrix}$.
Let $c_0=\epsilon^{-2}c \in R$ and $d_0=d-\epsilon^{-1}bc \in R$. Then $(c_0,d_0) \in \S$,
because $\alpha' c_0+\beta' d_0=1$ for $\alpha'=\epsilon\beta b+\epsilon^2 \alpha \in R$ and $\beta'=\beta \in R$. Moreover, for every $Z \in \HTwo$, we have
\begin{eqnarray*}
\Vert c_0 t(Z) +d_0 \Vert  =  \Vert\epsilon^{-2} c(\epsilon^{2} Z+\epsilon b)+d-\epsilon^{-1}bc \Vert =  \Vert cZ+d \Vert
\end{eqnarray*}
Hence $t(V_{c,d}) = V_{c_0,d_0}$.
\end{proof}

Similarly to Lemma~\ref{defside2}, one may analyse the intersection of three sides.

\begin{lemma}\label{intersect3sides}
The intersection of three distinct sides of $\F$ is of dimension at most $1$.
\end{lemma}

\begin{proof}
Let $S_1,\ S_2,\ S_3$ be three distinct sides of $\F$ and let $\gamma_i=\gamma_{S_i}$. By Lemma~\ref{ImageDFAlt}, each $S_i \subseteq V_{\gamma_i}$ and hence $V_{\gamma_i}$ is an essential hypersurface of $\F$. If $V_{\gamma_i}\ne V_{\gamma_j}$, for every $1 \leq i < j \leq 3$, then by Lemma~\ref{3inter}, $S_1 \cap S_2 \cap S_3$ is of dimension at most $1$.
Assume now that $V_{\gamma_1}=V_{\gamma_2}=V_{\gamma_3}=V_{c,d}$, for some $(c,d) \in \S$, $c \neq 0$. Then, by Lemma~
\ref{GammaVcd} and Lemma~\ref{ImageDFAlt}, $\gamma_1(S_1 \cap S_2 \cap S_3) = \F \cap \gamma_1(\F) \cap t_2\inv(\F) \cap t_3\inv(\F) \subseteq V_{-c,a} \cap V_{t_2} \cap V_{t_3}$, for some $a\in R$. If $V_{t_2} \neq V_{t_3}$, then by Lemma~\ref{3inter}, the last intersection is of dimension at most $1$ as desired. Suppose that $V_{t_2}=V_{t_3}=V_i^{\pm}$.
Using that $t_2 \neq t_3$ it is easy to prove that $i=3$ and hence $t_2,t_3\not\in \GEN{P_1,P_2}$.
By Lemma~\ref{GammaV3}, $t_3=Qt_2$ for $Q \in \langle P_1, P_2 \rangle$.
Thus $t_2\gamma_1(S_1 \cap S_2 \cap S_3) \subseteq t_2(V_{-c,a}) \cap \F \cap Q\inv(\F) \cap t_2(\F)$.
By Lemma~\ref{imageh}, $t_2(V_{-c,a}) =V_{c',d'}$ for some $(c',d') \in \S$, $c' \neq 0$ and by Lemma~\ref{ImageDFAlt} and the fact that $t_2 \not \in \langle P_1, P_2,\rangle$, $t_2(\F) \cap \F \subseteq \V_3^{\pm}$ and $Q\inv(\F_{\infty}) \cap \F_{\infty} \subseteq V_i^{\pm}$ with $i=1$ or $2$. Hence Lemma~\ref{3inter} allows to conclude.

Next assume that $V_{\gamma_1}=V_{\gamma_2}=V_{\gamma_3}=V_i^{\pm}$ with $i=1,2$ or $3$.
By Lemma~\ref{ImageDFAlt} and the fact that the $\gamma_i$'s are different we conclude that $i=3$ and hence $\gamma_1,\gamma_2,\gamma_3\in \Gamma_{\infty}\setminus \GEN{P_1,P_2}$.
Then, again by the same argument as above, $\gamma_2=Q_2\gamma_1$ and $\gamma_3=Q_3\gamma_1$ for $Q_2,Q_3 \in \langle P_1, P_2 \rangle$.
Thus $\gamma_1(S_1\cap S_2 \cap S_3) \subseteq \F  \cap \gamma_1(\F) \cap Q_2\inv(\F) \cap Q_3(\F) \subseteq V_3^{\pm} \cap V_i^{\pm} \cap V_j^{\pm}$.
with $i,j\in \{1,2\}$. However, as $Q_2\ne Q_3$, $V_i^{\pm}\ne V_j^{\pm}$ and again Lemma~\ref{3inter} allows to conclude.

So up to permutations of the sides, we are left to deal with the case $V_{\gamma_1}=V_{\gamma_2} \neq V_{\gamma_3}$ and four possible subcases:
\begin{enumerate}
\item $V_{\gamma_1}=V_{c_1,d_1}$ and $V_{\gamma_3}=V_{c_3,d_3}$ with $(c_1,d_1),(c_3,d_3)\in \S$ and $c_1$ and $c_3$ different from $0$;
\item $V_{\gamma_1}=V_{c_1,d_1}$ and $V_{\gamma_3} \in \V_{\infty}$ with $(c_1,d_1)\in \S$ and $c_1\ne 0$;
\item $V_{\gamma_1}\in \V_{\infty}$ and $V_{\gamma_3}=V_{c_3,d_3}$ for $(c_3,d_3) \in \S$ and $c_3 \neq 0$.
\item $V_{\gamma_1}, V_{\gamma_3} \in \V_{\infty}$.
\end{enumerate}
We deal with each case separately.

In the first case, by Lemma~\ref{GammaVcd}, $\gamma_2 = t\gamma_1$, for some $1 \neq t \in \h_{\infty}$.
Thus $\gamma_1(S_1 \cap S_2 \cap S_3) = \gamma_1(\F) \cap \F \cap t\inv(\F) \cap \gamma_1\gamma_3^{-1}(\F) \subseteq V_{\gamma_1\inv}\cap V_t \cap V_{\gamma_3\gamma_1\inv}$.
We claim that $V_{\gamma_1\inv}, V_t$ and $V_{\gamma_3\gamma_1\inv}$ are pairwise different.
On the one side $V_t\in \V_{\infty}$ while $V_{\gamma_1\inv}\in \V$ and hence  $V_{\gamma_1}\ne V_t$.
On the other side, since $V_{\gamma_1}=V_{c_1,d_1}$, the last row of $\gamma_1$ is of the form $(uc_1,ud_1)$ with $u\in \U(R)$, by Lemma~\ref{intermani}.
If $\gamma_3\gamma_1\inv\in \Gamma_{\infty}$ then the last row of $\gamma_3$ is of the same form and hence $V_{c_1,d_1}=V_{c_3,d_3}$ contradicting the hypothesis.
Thus $\gamma_1\gamma_3\inv\not\in \Gamma_{\infty}$ and therefore $V_{\gamma_3\gamma_1\inv}\ne V_t$.
Finally, if $V_{\gamma_1\inv}=V_{\gamma_3\gamma_1\inv}$ then the last rows of $\gamma_1\inv$ and $\gamma_3\gamma_1\inv$ differs in a unit and hence $\gamma_3 \in \Gamma_{\infty}$.

In the second case, again by the same reasoning, $\gamma_2=t\gamma_1$ for some $t \in \h_{\infty}$ and $\gamma_3 \in \h_{\infty}$.
Thus $S_1 \cap S_2 \cap S_3 = \F \cap \gamma_1^{-1}(\F) \cap \gamma_1^{-1}t^{-1}(\F) \cap \gamma_3\inv(\F)\subseteq V_{\gamma_1}\cap V \cap \left(\gamma_1^{-1}(\F) \cap \gamma_1^{-1}t^{-1}(\F) \right)$ with $V$ en element of $\V_{\infty}$. Observe that $\gamma_1\left(\gamma_1^{-1}(\F) \cap \gamma_1^{-1}t^{-1}(\F) \right) = \F \cap t^{-1}(\F) \subseteq V'$ for some element $V' \in \V_{\infty}$. Hence $S_1 \cap S_2 \cap S_3 \subseteq V_{\gamma_1} \cap V \cap \gamma_1^{-1}(V')$. By applying the implicit function theorem, as in the proof of Lemma~\ref{sidemani1} and distinguishing between the different cases for $V$ and $V'$, it is now a matter of a straightforward tedious calculation to prove that $V_{\gamma_1} \cap V \cap \gamma_1^{-1}(V')$ has dimension at most $1$.

In the third case $V_{\gamma_1}=V_3^{\pm}$ by (\ref{GammaV12}).
and $\gamma_2=Q\gamma_1\in \GEN{P_1,P_2}P_3^{\pm}$ for some $Q\in \GEN{P_1,P_2}$. Then $\gamma_1(S_1\cap S_2 \cap S_3) =\F \cap \gamma_1\inv(\F) \cap Q\inv (\F) \cap \gamma_1\gamma_3\inv(\F)\subseteq V_{\gamma_1\inv} \cap V_Q \cap V_{\gamma_3\gamma_1\inv}$. Since $\gamma_3\gamma_1\inv \not \in \h_{\infty}$ and $\gamma_1\inv\not \in \GEN{P_1,P_2}$ we deduce that $V_{\gamma_1\inv}$, $V_Q$ and $V_{\gamma_3\gamma_1\inv}$ are different and hence the result follows from Lemma~\ref{3inter}.

In the fourth case, again  $V_{\gamma_1}=V_{\gamma_2}=V_3^{\pm}$ and $\gamma_2=Q\gamma_1$, as above. Moreover $S_1\cap S_2\cap S_3=\F \cap \gamma_1^{-1}(\F) \cap Q\inv\gamma_1\inv(\F) \cap \gamma_3\inv(\F) \subseteq V_{\gamma_1} \cap V_{\gamma_3} \cap \left(\gamma_1^{-1}(\F) \cap Q\inv\gamma_1\inv(\F) \right)$, with $V_{\gamma_1}$ and $V_{\gamma_3}$ two different elements of $\V_{\infty}$. As in case 2, $\left(\gamma_1^{-1}(\F) \cap Q\inv\gamma_1\inv(\F) \right)= \gamma_1^{-1}(V')$ for some $V' \in \lbrace V_1^{\pm}, V_2^{\pm} \rbrace$. Again the implicit function theorem gives the desired result.
\end{proof}

We now describe $\F \cap V$ for $V \in  \V \cup \V_{\infty}$ in terms of intersection of tiles.

\begin{lemma}\label{variety_union_tiles}
Let $V \in \V \cup \V_{\infty}$ and set $\Gamma_{V}^{*}= \lbrace \gamma \in \Gamma\ :\ \F \cap \gamma(\F) \subseteq V \cap \F \rbrace$. Then
$$ V \cap \F = \bigcup_{\gamma \in \Gamma_V^*} \left(\F \cap \gamma(\F) \right).$$
\end{lemma}

\begin{proof}
One inclusion is obvious. To prove the other one, take $Z \in V \cap \F$. Suppose $V \in \V$, i.e. $V=V_{c,d}$ for some $(c,d) \in \S$. Take $\gamma=\begin{pmatrix} a & b \\ c & d \end{pmatrix} \in \Gamma$ for some $a,b \in R$. By (\ref{Height}), $h(\gamma(Z))=h(Z)$ and hence by Lemma~\ref{maxhorbit}, $\gamma(Z) \in \F_0$. As $\F_{\infty}$ is a fundamental domain for $\Gamma_{\infty}$, there exists $\tau \in \Gamma_{\infty}$ such that $\tau\gamma(Z) \in \F_{\infty}$. As, by (\ref{HeightInfinity}), $h(\tau\gamma(Z))=h(\gamma(Z))=h(Z)$, we have $\tau\gamma(Z) \in \F$ and thus $Z \in \F \cap \gamma^{-1}\tau^{-1}(\F)$. By Lemma~\ref{ImageDFAlt}, $\F \cap \gamma^{-1}\tau^{-1}(\F) \subseteq V_{c,d} \cap \F$.

If $V \in \V_{\infty}$ a similar reasoning may be applied.
\end{proof}

Next we show that every edge is contained in precisely two sides. To prove this we will make use of the following lemma.

\begin{lemma}\label{edgein2sides1}
Let $Z\in \F$ and let $C$ denote the intersection of the tiles of $\te$ containing $Z$ (i.e. $C=\bigcap_{Z \in \gamma(\F)} \gamma(\F)$).
If $\dim(C)=2$, then the number of sides of $\F$ containing $C$ is exactly 2.
\end{lemma}

\begin{proof}
By Lemma~\ref{intersect3sides}, $C$ is not contained in three different sides. So we only have to show that $C$ is contained in two different sides.

Clearly, every intersection of tiles containing $Z$ also contains $C$.
We claim that a similar property holds for the elements $M\in \M$.
Indeed, if $M\in \M$ then $M=\gamma(V)$ for some $V\in \V\cap \V_{\infty}$ and some $\gamma\in \h$.
If $Z\in M$ then $\gamma\inv(Z)\subseteq V$ and hence, by Lemma~\ref{variety_union_tiles}, there is $\gamma_1\in \Gamma$ such that  $\gamma\inv(Z) \in \F \cap \gamma_1(\F)\subseteq V$.
Therefore $Z\in \gamma(\F)\cap \gamma\gamma_1(\F)$ and hence $C\subseteq \gamma(\F)\cap \gamma\gamma_1(\F)\subseteq \gamma(V)=M$, as desired. We claim that a similar property holds for the elements $V\in \V \cup \V_{\infty}$.
Indeed, if $Z\in V$ then by Lemma~\ref{variety_union_tiles}, there is $\gamma_1\in \Gamma$ such that  $Z \in \F \cap \gamma_1(\F)\subseteq V$.
Hence $C\subseteq \F \cap \gamma_1(\F) \subseteq V$, as desired.

By Lemma~\ref{borderdefmani}, $Z$ is contained in at least one essential hypersurface $W$. Since $C$ has dimension 2, $Z$ cannot be contained in more than two elements of $\V\cup\V_{\infty}$, by Lemma~\ref{3inter}.
Then, the elements of $\V\cup \V_{\infty}$ containing $Z$ are essential hypersurfaces of $\F$ by Proposition~\ref{esshyploc3}. Take $\lambda >0$ such that $B(Z,\lambda) \cap \gamma(\F) = \emptyset$ for every $\gamma$ with $Z \not \in \gamma(\F)$ and $B(Z, \lambda) \cap W =
\emptyset$ for every essential hypersurface $W$ of $\F$ not containing $Z$. Again by Proposition~\ref{esshyploc3}, if $W$ is an essential hypersurface of $\F$ containing $Z$ then
$B(Z, \lambda) \cap W \cap \F$ is of dimension $3$ and $B(Z, \lambda) \cap W \cap \F \subseteq \bigcup_{1\ne \gamma\in \h, Z \in \gamma(\F)} (\F \cap \gamma(\F) \cap W)$
and thus one of these intersections is of dimension $3$.
Therefore, if $W$ is an essential hypersurface of $\F$ containing $C$ then it contains a side containing $Z$.
Thus, if $C$ is contained in exactly two elements of $\V \cup \V_{\infty}$ then each contains one side containing $Z$ and these two sides have to be different by Lemma~\ref{sidemani1}.

Hence we may assume that $Z$ is contained in exactly one element $W$ of $\V\cup \V_{\infty}$.
Then $C$ is contained in at least one side $S_1$ and $S_1\subseteq W$.
Let $\gamma_1 \in \Gamma_{S_1}$ (i.e. $S_1=\F \cap \gamma_1\inv(\F)$). Then, by Lemma~\ref{ImageDFAlt}, $W=V_{\gamma_1}$.
As $\dim(C)=2$, $\dim(S_{\gamma_1})=3$ and $C\subseteq \F\cap S_1$ there is $\gamma_2\in \h\setminus \{1,\gamma_1\}$ with $Z\in \gamma_2\inv(\F)$.
Thus $\gamma_2=t\gamma_1$ with $t\in \h_{\infty}\setminus \{1\}$ and if $\gamma_1\in \Gamma_{\infty}$ then $t\in \GEN{P_1,P_2}$, by Lemma~\ref{GammaV3} and Lemma~\ref{GammaVcd}.
We consider separately two cases.

First assume that $W=V_i^{\pm}$  with $i=1$ or $2$ and without loss of generality, we may assume that $W=V_1^+$.
Then $\gamma_1=P_1\inv$ by (\ref{GammaV12}) and $\gamma_2=P_1^{\alpha}P_2^{\beta}$ for some integers $\alpha$ and $\beta$.
Moreover $Z=(\frac{1}{2},s_2,r,h)$ and $(\alpha+\frac{1}{2},s_2+\beta,r,h) = \gamma_2(Z) \in \F$.
Thus $\alpha\in \{0,-1\}$.
However, if $\alpha=0$ then $V_{\gamma_2}=V_2^{\pm}\ne W$.
Thus $\alpha=-1$ and hence $\beta=\pm 1$, because $\gamma_1\ne \gamma_2$.
Then $Z\in V_2^{\pm}$, a contradiction.

Assume now that either $W\in \V$ or $W=V_3^{\pm}$.
In both cases $V_t\ne V_{\gamma_1}$, because in the first case $V_{\gamma_1}\in \V$ and $V_t\in \V_{\infty}$ and in the second case $V_t=V_i^{\pm}$ with $i=1$ or $2$.
Moreover $\gamma_1(C)\subseteq \F \cap \gamma_1(\F) \cap t\inv(\F) \subseteq V_{\gamma_1\inv} \cap V_t$. As $\F\subseteq E_t$, $\gamma_1(S_{\gamma_1})=S_{\gamma_1^{-1}} \subseteq E_t$.
Let $M=\gamma_1\inv(V_t)$, $M^{\ge}=\gamma_1\inv(E_t)$ and $M^{\le}=\gamma_1\inv(E_t')$. Then, by the previous, $S_{\gamma_1} \subseteq M^{\ge}$.
By the choice of $\lambda$ and Lemma~\ref{sidemani1}, $B(Z,\lambda)\cap W \cap \F \cap M^{\le}$ has dimension 3.
Thus, by Lemma~\ref{borderunionsides}, $Z \in S$ for some side $S \subseteq W$ and different from $S_{\gamma_1}$. Hence also $C$ is contained in two different sides.
\end{proof}

The following is a consequence of Proposition~\ref{relintside} and Lemma~\ref{edgein2sides1}.

\begin{corollary}\label{edgein2sides2}
If $E$ is an edge of the fundamental domain $\F$, then there are precisely two sides that contain $E$.
\end{corollary}

The following lemma is obvious.

\begin{lemma}\label{edgeinter}
Let $E_1$ and $E_2$ be two different edges of some tile. Then the intersection $E_1 \cap E_2$ is of dimension at most $1$.
\end{lemma}


Finally, in order to be able to  describe the relations, we need two more lemmas.

\begin{lemma}\label{edgeunion}
Let $\gamma_1,\gamma_2 \in \Gamma$. Assume $\gamma_1(\F) \cap \gamma_2(\F)$ has dimension $2$. Then,
\begin{enumerate}
\item there exists $Z_0 \in \gamma_1(\F) \cap \gamma_2(\F)$ such that $\bigcap_{Z_0 \in \gamma(\F)} \gamma(\F)$ is of dimension $2$; and
\item for every such $Z_0$, the set  $\bigcap_{Z_0 \in \gamma(\F)} \gamma(\F)$ is an edge contained in $\gamma_1(\F) \cap \gamma_2(\F)$.
\end{enumerate}
\end{lemma}

\begin{proof}
1. We first show the existence of a point $Z_0 \in \gamma_1(\F) \cap \gamma_2(\F)$ such that $\bigcap_{Z_0 \in \gamma(\F)} \gamma(\F)$ is of dimension $2$.
For every $Z \in \gamma_1(\F) \cap \gamma_2(\F)$, let $\Gamma_Z=\lbrace \gamma \in \Gamma\setminus \lbrace \gamma_1, \gamma_2 \rbrace :\ Z \in \gamma(\F) \rbrace$.
We claim that $\Gamma_Z\ne\emptyset$ for every $Z\in \gamma_1(\F)\cap \gamma_2(\F)$.
Otherwise there is a $\lambda>0$ such that $B(Z,\lambda)\cap \gamma(\F)=\emptyset$ for every $\gamma\in \Gamma\setminus \{\gamma_1,\gamma_2\}$.
Then $B(Z,\lambda)\subseteq \gamma_1(\F)\cup \gamma_2(\F)$, which is in contradiction with Lemma~\ref{connected}.
This proves the claim.
Thus $\gamma_1(\F) \cap \gamma_2(\F) = \bigcup_{Z \in \gamma_1(\F) \cap \gamma_2(\F)} \left( \bigcap_{\gamma \in \Gamma_Z} \gamma(\F) \cap \gamma_1(\F) \cap \gamma_2(\F) \right)$. As, by assumption, $\gamma_1(\F) \cap \gamma_2(\F)$ has dimension $2$ and $\Gamma$ is countable, it follows that there exists  $Z_0 \in \gamma_1(\F) \cap \gamma_2(\F)$ with $\bigcap_{Z_0 \in \gamma(\F)} \gamma(\F)$ of dimension $2$.

2. Since $Z_0$ belongs to only finitely many tiles (by Lemma~\ref{locallyfinite}), say $\gamma_1(\F),\gamma_2(\F),
\ldots , \gamma_n(\F)$, we have that $Z_0 \in \bigcap_{i=1}^n \gamma_i(\F)$ and $\bigcap_{i=1}^n \gamma_i(\F)$ has
dimension $2$. We want to prove that this intersection is an edge.  Let $\gamma_0 \in \Gamma \setminus \lbrace \gamma_1,
\ldots , \gamma_n \rbrace$. As $Z_0 \not \in \gamma_0(\F)$, it is clear that $\bigcap_{i=1}^n \gamma_i(\F) \not
\subseteq \gamma_0(\F)$. Hence it remains to prove that $\bigcap_{i=1}^n \gamma_i(\F)$ intersects $\gamma_0(\F)$ in
dimension at most $1$. Suppose this is not the case, i.e. $\bigcap_{i=0}^n \gamma_i(\F)$ is of dimension $2$. As in the
first part of the proof, there exists $Z_1 \in \bigcap_{i=0}^n \gamma_i(\F)$ such that $\bigcap_{Z_1 \in
\gamma(\F)}\gamma(\F)$ is of dimension $2$. Let $\gamma_0, \ldots, \gamma_m$ be all elements of $\Gamma$ (with $m \geq
n$) such that $Z_1 \in \gamma_i(\F)$. So $\bigcap_{i=0}^m \gamma_i(\F)$ is of dimension $2$. By Corollary~\ref{boule},
let $\lambda_1 >0$ and $B=\overline{B(Z_1,\lambda_1)}$ be such that $B \cap \gamma(\F) \neq \emptyset$ if and
only if $Z_1 \in \gamma(\F)$. Then $B = \left(B \cap \bigcup_{i=1}^n \gamma_i(\F) \right) \cup \left(
B \cap \left( \bigcup_{i=n+1}^m \gamma_i(\F) \cup \gamma_0(\F) \right) \right)$, where both factors are closed
sets of dimension $4$ (by Lemma~\ref{Floc4boundloc3}). Hence by Lemma~\ref{connected}, $\left( B \cap
\bigcup_{i=1}^n \gamma_i(\F) \right) \cap \left( B \cap \left( \bigcup_{i=n+1}^m \gamma_i(\F) \cup \gamma_0(\F)
\right) \right)$ is of dimension $3$. Thus there exists $1 \leq j_1 \leq n$ and $n+1 \leq j_2 \leq m$ or $j_2=0$ such
that $Z_1 \in \gamma_{j_1}(\F) \cap \gamma_{j_2}(\F)$ and the latter intersection is of dimension $3$. Hence
$\gamma_{j_1}(\F) \cap \gamma_{j_2}(\F)$ is a side of the tile $\gamma_{j_1}(\F)$. We now come back to $Z_0$. Let
$\lambda_0 >0$ and $B'=\overline{B(Z_0,\lambda_0)}$ be such that $\cap \gamma(\F) \neq \emptyset$ if and only
if $Z_0 \in \gamma(\F)$. Thus $B'= \left(B'\cap \gamma_{j_1}(\F) \right) \cup \bigcup_{i=1, i \neq
j_1}^n \left(B' \cap \gamma_i(\F) \right)$ and again by Lemma~\ref{connected}, there exists $1 \leq j_3 \leq n$
and $j_3 \neq j_1$ such that $\gamma_{j_1}(\F) \cap \gamma_{j_3}(\F)$ is of dimension $3$. Thus it is a side of the tile
$\gamma_{j_1}(\F)$. Hence $Z_0 \in \partial \gamma_{j_1}(\F)$, such that $\bigcap_{Z_0 \in \gamma(\F)} \gamma(\F)$ is of
dimension $2$. Moreover the latter is contained in $\gamma_{j_1}(\F) \cap \gamma_{j_3}(\F)$, which is a side. Hence by
Lemma~\ref{edgein2sides1}, there exists a second side of $\gamma_{j_1}(\F)$, which contains $\bigcap_{Z_0 \in
\gamma(\F)} \gamma(\F)$. Thus this side also contains $Z_0$ and hence there exists $1
\leq j_4 \leq n$ with $j_4 \neq j_1$ and $j_4 \neq j_3$ such that $\gamma_{j_1}(\F) \cap \gamma_{j_4}(\F)$ is a side.
Thus the tile $\gamma_{j_1}(\F)$ contains three sides, $\gamma_{j_1}(\F) \cap \gamma_{j_2}(\F)$, $\gamma_{j_1}(\F) \cap
\gamma_{j_3}(\F)$ and $\gamma_{j_1}(\F) \cap \gamma_{j_4}(\F)$ and their intersection contains  $\bigcap_{i=0}^m
\gamma_i(\F)$, which is of dimension $2$. This contradicts Lemma~\ref{intersect3sides}. Hence $\bigcap_{i=1}^n
\gamma_i(\F)$ is an edge and it contains $Z_0$.

\end{proof}

\begin{lemma}\label{seqassedge}
Let $E$ be an edge.
The finitely many elements $\gamma \in \h$ with $E \subseteq \gamma(\F)$ can be ordered, say as $\gamma_0,\gamma_1, \ldots , \gamma_m=\gamma_0$, such that $\gamma_{j-1}(\F) \cap \gamma_j(\F)$ is a side (containing $E$) for every $1 \leq j \leq m$.
Moreover,  up to cyclic permutations and reversing the ordering, there is  only one possible ordering with this property.
\end{lemma}

\begin{proof}
Recall that there are only finitely many $\gamma \in \Gamma$ with $E \subseteq \gamma(\F)$. Let $\gamma_0$ be such an element.
Then $E$ is an edge of $\gamma_0(\F)$ and hence, by Corollary~\ref{edgein2sides2}, there exists two sides, say $\gamma_0(\F)\cap \gamma_1{\F}$ and $\gamma_0(\F)\cap \gamma_{m-1}(\F)$ of $\gamma_0(\F)$
containing $E$.
Now $E$ also is an edge of $\gamma_1(\F)$ and $\gamma_0(\F)\cap \gamma_1(\F)$ is one of the two sides of $\gamma_1(\F)$ containing $E$. Hence there exists a third tile, say $\gamma_2(\F)$,
such that $\gamma_1(\F) \cap \gamma_0(\F)$ and $\gamma_1(\F)
\cap \gamma_{2}(\F)$ are the two different sides of $\gamma_1(\F)$ containing $E$. So $\gamma_2 \not \in \lbrace \gamma_0, \gamma_1 \rbrace$.
One may continue this process and have a sequence $\gamma_{-1},\gamma_0,\gamma_1,\gamma_2,\dots$ of elements of $\Gamma$ such that  $\gamma_{i-1}(\F)\cap \gamma_i(\F)$ and $\gamma_i(\F)\cap \gamma_{i+1}(\F)$ are the two sides of $\gamma_i(\F)$ containing $E$ for every $i\ge 0$.
In particular every three consecutive elements of the list of $\gamma_i$'s are different.
As there are only finitely many tiles containing
$E$, after finitely many steps we obtain $\gamma_j \in \Gamma$ with $\gamma_i=\gamma_j$ and $0\le i<j$.
Let $i$ be minimal with this property. We claim that $i=0$. Indeed, if $0 <i$, then by construction, $\gamma_i(\F) \cap \gamma_{i-1}(\F)$, $\gamma_i(\F) \cap
\gamma_{i+1}(\F)$, $\gamma_j(\F) \cap \gamma_{j-1}(\F)=\gamma_i(\F) \cap \gamma_{j-1}(\F)$ and $\gamma_j(\F) \cap \gamma_{j+1}(\F) = \gamma_i(\F) \cap \gamma_{j+1}
(\F)$ are sides of $\gamma_i(\F)$, all containing $E$.  By Corollary~\ref{edgein2sides2}, $\gamma_{i-1}=\gamma_{j-1}$ and
$\gamma_{i+1}=\gamma_{j+1}$ or $\gamma_{i+1}=\gamma_{j-1}$ and $\gamma_{i-1}=\gamma_{j+1}$. Both cases contradict the minimality of $i$. It remains
to prove that $\lbrace \gamma_0, \ldots , \gamma_{j-1} \rbrace = \lbrace \gamma \in \Gamma :\ E \subseteq \gamma(\F) \rbrace$. This may be easily done by arguments similar as the arguments used in the end of the proof of Lemma~\ref{edgeunion}.

To prove the last part, notice that instead of starting with the chosen element $\gamma_0$ one could have  started with any of  the finitely many elements $\gamma \in \Gamma$ such that $E \subseteq \gamma(\F)$. Second, note that once the element $\gamma_0$ is fixed, there exists two unique tiles $\gamma_{m-1}(\F)$ and $\gamma_1(\F)$ such that $\gamma_0(\F) \cap \gamma_{m-1}(\F)$ and $\gamma_0(\F) \cap \gamma_1(\F)$ are sides. Hence, up to a choice of the first element $\gamma_0$, thus up to a cyclic permutations, and up to a choice of a second element, thus up to reversing ordering, there is only one possible ordering.
\end{proof}

Based on the previous lemma, we give the following new definition.

\begin{definition}
Let $E$ be an edge.
We call an ordering $(\gamma_0, \ldots, \gamma_{m-1},\gamma_0)$ of the elements $\gamma\in \h$ such that $E\subseteq \gamma(\F)$ as in Lemma~\ref{seqassedge} an \emph{edge loop} of $E$.
\end{definition}

We now come to a description of the relations of $\h$. We first define what is called a cycle. To do so, we  fix some notations.

\begin{definition}\label{cycle}
Let $E$ be an edge of $\F$ and $S$ a side of $\F$ containing $E$. Then define recursively the sequence $(E_1,S_1,E_2,S_2, \ldots)$ as follows:
\begin{enumerate}
\item $E_1=E$ and $S_1=S$, \label{cond1}
\item $E_{i+1}=\gamma_{S_i}(E_i)$, \label{cond2}
\item $S_{i+1}$ is the only side of $\F$ different from $S_i^*$ that contains $E_{i+1}$. \label{cond3}
\end{enumerate}
\end{definition}

Note that parts $2$ and $3$ are justified by Lemma~\ref{edgeonedge} and Corollary~\ref{edgein2sides2}. For every $i$, clearly $E_i \subseteq S_i$ and $E_{i+1} \subseteq S_{i}^{*}$, as $\gamma_{S_i}(S_i)=S_i^{*}$. Moreover, each pair $(E_i, S_i)$ determines the pairs $(E_{i-1}, S_{i-1})$ and $(E_{i+1}, S_{i+1})$. This is clear for the subsequent pair $(E_{i+1}, S_{i+1})$ but also for the previous one $(E_{i-1},S_{i-1})$ because if $i \geq 2$, then $S_{i-1}^{*}$ is the only side of $\F$ containing $E_i$ and different from $S_{i}$ and $E_{i-1}=\gamma_{S_{i-1}}^{-1}(E_i)=\gamma_{S_{i-1}^{*}}(E_i)$. Thus each pair $(E_i, S_i)$ determines the sequence $(E_1, S_1, \ldots )$.

Now we relate the sequence $(E_1,S_1,E_2,S_2,\dots)$ with an edge loop of $E$.
Let $L_{\F}=\{\gamma\in \h : E\subseteq \gamma(\F)\}$. Clearly $1,\gamma_{S_1}\inv \in L_{\F}$ and as $S_1$ is a side of $\F$, there is a unique edge loop of $E$  of the form $(\gamma_0=1,\gamma_1=\gamma_{S_1},\gamma_2,\dots,\gamma_{m-1},\gamma_0)$. As $S_1^*=S_{\gamma_1}$ and $S_2$ are the two sides of $\F$ containing $E_2$, $S_1=\gamma_1\inv(S_1)^*$ and $\gamma_1\inv(S_2)$ are the two sides of $\gamma_1\inv(\F)$ containing $\gamma_1\inv(E_2)=E$. Therefore $\gamma_1\inv(S_2)=\gamma_1\inv(\F)\cap \gamma_2(\F)$ and hence $S_2=\F\cap \gamma_1\gamma_2(\F)$. A similar argument shows that $S_i=\F \cap \gamma_1 \gamma_2 \dots \gamma_i(\F)$ for every $i$.
In particular, the sequence $(E_1,S_1,E_2,S_2,\dots)$ is periodic, i.e. there is $n >0$ such that $(E_{i+n},S_{i+n})=(E_i,S_i)$ for every $i$. If $n$ is minimal with this property, then $(E_i,S_i) \neq (E_j,S_j)$ for $1 \leq i<j\leq n$. We call $(E_1,S_1,\ldots E_n,S_n)$ the \emph{cycle} determined by $(E_1,S_1)$. Hence this proves the following lemma.

\begin{lemma}\label{cyclefinite}
Let $E_1$ be an edge of $\F$ and $S_1$ a side of $\F$ containing the edge $E_1$. The cycle starting with $(E_1,S_1,\ldots)$ is a   finite cycle.
\end{lemma}

Note that $S_1$ and $S_n^{*}$ are the two different sides of $\F$ containing $E_1$ and thus there are two cycles starting with the edge $E_1$, namely $(E_1,S_1,\ldots E_n,S_n)$ and $(E_1,S_n^{*}, E_{n}, S_{n-1}^{*}, \ldots , E_2,S_1^{*})$.

It is now also clear that if $E$ is an edge and $S$ is a side containing $E$, then all the cycles containing $E$ are cyclic permutations of the cycle starting with $(E,S)$ and the cycles obtained by replacing in those the sides by their paired sides and reversing the order. In particular, if $E_i=E_j$ with $1 \leq i < j \leq n$, then $S_i \neq S_j$ and hence
\begin{eqnarray*}
 & &(E_i, S_i, E_{i+1}, S_{i+1}, \ldots, E_n, S_n, E_1, S_1, \ldots, E_{i-1}, S_{i-1}) \\
& = &(E_j, S_{j-1}^{*}, E_{j-1}, S_{j-2}^{*}, \ldots, S_2^{*}, E_1, S_n^{*}, E_n, S_{n-1}^{*}, \ldots, E_{j+1}, S_{j}^{*}).
\end{eqnarray*}

\begin{lemma}\label{cyclefinorder}
If $(E_1,S_1,E_2,S_2, \ldots, E_n,S_n)$ is a cycle of $\F$ then $\gamma_{S_n}\gamma_{S_{n-1}},\ldots , \gamma_{S_1}$ has finite order.
\end{lemma}
\begin{proof}
Let $\gamma=\gamma_{S_n}\gamma_{S_{n-1}},\ldots , \gamma_{S_1}$. Clearly $\gamma(E_1)=E_1$ and thus $\gamma^k(E_1)=E_1$ for all non-negative integers $k$.
Hence $E_1\subseteq \gamma^k(\F)$ and because every edge is contained in only finitely many tiles, $\gamma$ has finite order.
\end{proof}

Because of Theorem~\ref{Poincare}, we thus obtain a natural group epimorphism
\begin{equation}\label{delta}
\varphi: \Delta \rightarrow \h: \left[\gamma_S\right] \mapsto \gamma_S
\end{equation}
where $\Delta$ is the group given by the following presentation
\begin{itemize}
\item \emph{Generators}: a generator $\left[\gamma_S\right]$ for each side $S$ of $\F$
\item \emph{Relations}: $\left[\gamma_S\right]\left[\gamma_S*\right]=1$ if the sides $S$ and $S^*$ are paired and $(\left[ \gamma_{S_n}\right]\ldots \left[\gamma_{S_1}\right])^m=1$ if $(E_1,S_1,E_2,S_2, \ldots, E_n,S_n)$ is a cycle and $m$ is the order of $\gamma_{S_n}\ldots \gamma_{S_1}$.
\end{itemize}

Our next aim is to show that $\varphi$ is also injective and thus we obtain a presentation for $\Gamma$. To that end we introduce the following definition.

\begin{definition}
A \emph{loop of tiles} is a finite list of tiles $( h_0(\F), h_1(\F), \ldots , h_n(\F) )$ with $h_i \in \h$ such that $h_0(\F) = h_n(\F)$ and $h_{i-1}(\F) \cap h_i(\F)$ is of dimension $3$, for every $i=1, \ldots n$.
\end{definition}

The last condition means that $\F \cap h_{i}^{-1}h_{i-1}(\F)$ is of dimension $3$, which is equivalent to $\gamma_i=h_{i-1}^{-1}h_{i}$ being a side-paring transformation. Moreover we get the relation $\gamma_1 \gamma_2\ldots \gamma_n=1$ which is called a \emph{loop relation}.
In Theorem~\ref{Poincarerelations} we will show that these relations form a complete set of relations of $\h$.

It is easy to see that the pairing and cycle relations are determined by  loop relations. Indeed, let $S$ and $S^*$ be two paired sides of $\F$. Then $( \F, \gamma_s(\F), \gamma_S\gamma_{S^*}(\F)=\F )$ is a loop of tiles which gives as loop relation the pairing relation. Let $(E_1,S_1,E_2,S_2, \ldots, E_n,S_n)$ be a cycle. We have seen that there exists a positive integer $m$ such that $(\gamma_{S_n}\gamma_{S_{n-1}}\ldots  \gamma_{S_1})^m=1$.
Set $h_0=1$ and $h_i=h_{i-1}\gamma_{S_j}$ where $j \equiv i \mod (n)$ and $i \in \lbrace 1, \ldots , mn \rbrace$. Consider $( h_0(\F), \ldots, h_{mn}(\F) )$. Clearly $h_{mn}=1$ and hence $h_{mn}(\F)=\F=h_0(\F)$. Also, for every $i \in \lbrace 1, \ldots, mn \rbrace$, $h_{i-1}(\F) \cap h_i(\F) = h_{i-1}(\F \cap \gamma_{S_j}(\F))$, where $\gamma_{S_j}$ is a side-paring transformation and hence $\F \cap \gamma_{S_j}(\F)$ and thus also $h_{i-1}(\F) \cap h_i(\F)$ is of dimension $3$. So, by definition, $( h_0(\F), \ldots, h_{mn}(\F) )$ is a loop of tiles and the associated loop relation is the cycle relation.


Consider the union of intersections of varieties of $\M$ such that these intersections have dimension at most $1$. Let $\Omega$ denote  the complement of this set in $\HTwo$. Note that by Lemma~\ref{pathconnected}, $\Omega$ is path-connected.

In the remainder we fix $\alpha$ to be a continuous map
  $$\alpha: \left[0,1 \right] \rightarrow \Omega ,$$
such that
 $$\alpha(0) \in g(\F)^{\circ} \mbox{ and } \alpha(1) \in g'(\F)^{\circ} ,$$
for some $g,g' \in \h$ and such that $\alpha$ is made up of a finite number of line segments, which are parametrized by a polynomials of  degree at most one.  Moreover, for  each line segment forming $\alpha$, we suppose that at least one of its  end-points does not belong to any element in $\M$.
Note that  such a map exists. Indeed, let $\M_{\alpha}$ denote  the set consisting of the elements of $\M$ that have non-empty intersection with (the image of)  $\alpha$. Then it is easy to see that $\Omega \setminus \bigcup_{M \in \M_{\alpha}} M$ is dense in $\Omega$. Hence, by Lemma~\ref{Eisele1}, such a map $\alpha$ indeed  exists.

\begin{lemma}\label{alfcutfinite}
The set $\{t\in [0,1] : \alpha(t)\in \partial g(\F) \text{ for some } g\in \h\} $ is finite.
\end{lemma}

\begin{proof}
We know from Lemma~\ref{locallyfinite}  that the compact set $\alpha(\left[0,1\right])$ only intersects finitely many tiles and thus also only finitely many sides. By Lemma~\ref{ImageDFAlt}, for every side $S$ of some tile $g(\F)$, $g(\F) \cap S$ is contained in a precise variety $M \in \M$.
As the elements of $\M$ are real semi-algebraic varieties, which are given by polynomials, a line segment is either contained in such a variety or it intersects it in finitely many points.
The path $\alpha$ consists of finitely many line segments, such that at least one of the end-points of these line segments does not belong to any element in $\M$.
If  such a line segment  $l$ is  parametrized by a polynomial of first degree then there are only finitely many $t\in [0,1]$ such that $l(t)\in M$ for some $M\in  \M$.
Moreover, if a line segment $l$ is  parametrized by a polynomial of degree $0$, i.e. the image of $l$ is  just a point, then by definition this point is not contained in any element of $\M$ and $l \cap M = \emptyset$ for every $M \in \M$. Thus $\alpha(t)$ belongs to a side for only finitely many $t$. Hence the lemma follows by Lemma~\ref{borderunionsides}.
\end{proof}


\begin{lemma}\label{alfadaptexists}
Let $\alpha$ be a continuous map $[0,1] \rightarrow \Omega$ such that
\begin{enumerate}
\item $\alpha(0) \in g(\F)^{\circ}$ and $\alpha(1) \in g'(\F)^{\circ}$ for some $g,g' \in \h$,
\item $\alpha$ is made up of a finite number of line segments,
\item for each of these line segments at least one of its two end-points does not belong to any element in $\M$.
\end{enumerate}
Then there exists a unique ordered list $\mathcal{L}=(a_0,g_1,a_1,g_2,a_2 \ldots , g_n, a_n)$, where $g_i \in \h$, $0=a_0 < a_1 < \ldots < a_n=1$ and for every $1\leq i \leq n$,
\begin{enumerate}[(i)]
\item $g_{i-1} \neq g_i$,
\item $\alpha([a_{i-1},a_i])\subseteq g_i(\F)$ and
\item there exists $\epsilon_0>0$, such that $\alpha((a_i,a_i + \epsilon)) \cap g_i(\F) = \emptyset$.
\end{enumerate}
We call $\mathcal{L}$ the \emph{partition of $\alpha$}.
\end{lemma}

\begin{proof}
As $\alpha(0) \in g(\F)^{\circ}$, we set $g_1=g$ and $a_0=0$.
By Lemma~\ref{alfcutfinite}, there are only finitely many $t \in \left[0,1\right]$ such that $\alpha(t) \in \partial g(\F)$ for some $g \in \h$, say $t_1 < t_2 < \ldots < t_m$.
Put  $t_0=0$ and $t_{m+1}=1$. For  each $1 \leq i \leq m+1$, the set  $\alpha(t_{i-1},t_i)$ is contained in the interior of only one tile, say  $h_i(\F)$ with $h_i\in \Gamma$.
We now construct recursively $(a_0=0=t_0,g_1=g,a_1,g_1,\ldots,g_n,a_n)$. Let $a_1=t_i$ with $t_i$ maximal such that   $\alpha(\left[0,t_i\right]) \subseteq g_{1}(\F)$. Then $h_1=\ldots=h_i=g_1$,  $\alpha(\left[t_i,t_{i+1}\right]) \subseteq h_{i+1}(\F)$ and  we set $g_2=h_{i+1}$. Assume we have constructed $(a_0,g_1,a_1,\ldots , a_k,g_k)$ satisfying conditions (i)-(iii) and such that $a_k=t_i$ for some $i$ and $g_k=h_{i+1}$. Then let $a_{k+1}=t_j$ with $j$ maximal with $\alpha(\left[t_i,t_j\right]) \subseteq g_k(\F)$ and $g_{k+1}=h_{j+1}$. After finitely many steps, we obtain an ordered list $(a_0=0=t_0,g_1=g,a_1,g_1,\ldots,g_n,a_n)$ satisfying conditions  (i)-(iii). Clearly  such a  sequence is unique.
\end{proof}

\begin{remark}
As, by assumption, $\alpha(0) \in g(\F)^{\circ}$ and $\alpha(1) \in g'(\F)^{\circ}$ for some $g,g' \in \h$, the first element $g_1$ of the partition of $\alpha$ equals $g$ and the last element $g_n$ equals $g'$.
\end{remark}

%
The remainder of the paper is based on ideas from a recent proof of the presentation part of the classical Poincar\'e theorem \cite{poincarepaper}.
Let $g,h\in \Gamma$. Let $C$ be a cell of $\te$ of  dimension $m\geq 2$  and that is contained in $g(\F)\cap h(\F)$.
We define $\kappa_C (g,h)\in \Delta$ as follows.
\begin{itemize}
\item If $m=4$ then $\kappa_C(g,h)=1$.
\item If $m=3$ then $\kappa_C (g,h) =[g\inv h]$.
\item If $m=2$ then $C$ is an edge contained in $g(\F)\cap h(\F)$ and thus, by Lemma~\ref{seqassedge}, $g$ and $h$ belong to an  edge loop of $C$. Up to a cyclic permutation, we can write the edge loop of $C$ as  $(g= k_0, \ldots , k_t =h, k_{t+1}, \ldots , k_m=g)$ (or the equivalent edge loop $(g=k_m,k_{m-1},\dots,k_t=h,k_{t-1},\dots,k_1,k_0=g)$) and we set
	$$\kappa_C (g,h) =[k_0\inv k_1][k_1\inv k_2]\;  \cdots [k_{t-1}\inv k_t] = [k_m\inv k_{m-1}] \cdots [k_{t+1}\inv k_{t}].$$
\end{itemize}

Observe that $\kappa_{C}(g,g)=1$ in the three cases.

\begin{lemma}\label{kappaElementaryL}
Let $g,h\in \Gamma$ and let $C$ be a cell of $\te$ of dimension $m\ge 2$ and that is contained in $g(\F)\cap h(\F)$. The  following properties hold.
\begin{enumerate}
\item\label{kappaAntisymmetric} $\kappa_C(g,h)=\kappa_C(h,g)\inv$.
\item\label{kappaIndependence} If $D$ is cell of $\te$ contained in $C$ and of dimension at least $2$ then $\kappa_D(g,h)=\kappa_C(g,h)$.
\item\label{kappaTransitiveCell} If $g_1,\dots,g_n\in \Gamma$ and $C\subseteq \bigcap_{i=1}^n g_i(\F)$ then
    $\kappa_C(g_1,g_n)=\kappa_C(g_1,g_2)\kappa_C(g_2,g_3)\cdots \kappa_C(g_{n-1},g_n)$.
\end{enumerate}
\end{lemma}

\begin{proof}
\ref{kappaAntisymmetric}.
If $m=4$, then $g=h$ and there is nothing to prove.
If $m=3$ then  $g^{-1}h$ and $h^{-1}g$ are pairing transformations and hence by the pairing relations of the group $\Delta$, $\kappa_C(g,h)\kappa_C(h,g)=1$. Finally, if $m=2$, then we can write the edge loop of $C$ as $(g= k_0, \ldots , k_t =h, k_{t+1}, \ldots , k_m=g)$ and thus
\begin{eqnarray*}
\kappa_C (g,h) &  = & [k_0\inv k_1][k_1\inv k_2]\;  \cdots [k_{t-1}\inv k_t], \\
\kappa_C (h,g) & = &[k_t\inv k_{t+1}][k_{t+1}\inv k_{t+2}]\;  \cdots [k_{m-1}\inv k_m].
\end{eqnarray*}
It is now easy to see that $\kappa_C(g,h)\kappa_C(h,g)=1$ and hence the result follows.

\ref{kappaIndependence}. If $C=D$ then there is nothing to prove. So assume that $C\ne D$. If $C$ is a side then $D$ is an edge and $g$ and $h$ are two consecutive elements of the edge loop of $D$. Then $\kappa_D(g,h)=[g\inv h]=\kappa_C(g,h)$. Otherwise, $C$ is a tile and hence $g=h$. Thus $\kappa_D(g,h)=1=\kappa_C(g,h)$.

\ref{kappaTransitiveCell}. By induction it is enough to prove the statement for $n=3$. So assume $n=3$. If either $g_1=g_2$ or $g_2=g_3$ then the desired equality is obvious. So assume that $g_1\ne g_2$ and $g_2\ne g_3$.
If $C$ is an edge then, up to a cyclic permutation, possibly  reversing the order and making use of
Lemma~\ref{seqassedge}, the edge loop of $C$ is of the form $(g_1=k_0,\dots,g_2=k_t,\dots,g_3=k_l,\dots,k_m=g_1)$. Then,
\begin{eqnarray*}
    \kappa_C(g_1,g_3)&=&[k_0\inv k_1][k_1\inv k_2]\cdots [k_{l-1}\inv k_l] \\
   &=& ([k_0\inv k_1][k_1\inv k_2]\cdots [k_{t-1}\inv k_t]) \; ( [k_t\inv k_{t+1}] \cdots [k_{l-1}\inv k_l]) \\
   &=& \kappa_C(g_1,g_2)\kappa_C(g_2,g_3)
\end{eqnarray*}
Otherwise, $S=g_1\inv(C)$ is a side of $\F$, $g_S=g_1\inv g_2$, $g_{S'}=g_2\inv g_1$ and $g_1=g_3$.
Then,
	$$\kappa_C(g_1,g_3) = 1 = [g_S][g_{S'}] = \kappa_C(g_1,g_2)\kappa_C(g_2,g_3).$$
\end{proof}

Let $Z \in \Omega$, $g,h\in \h$ and $Z \in C\subseteq g(\F)\cap h(\F)$ for some cell $C$. Then, by the definition of $\Omega$, the dimension of $C$ is at least $2$ and we define
$$\kappa_Z(g,h)=\kappa_C(g,h).$$
This is well defined. Indeed, suppose $D$ is another cell containing $Z$ and contained in $g(\F)\cap h(\F)$ with $\kappa_C(g,h)\ne \kappa_D(g,h)$. Then $g\ne h$ and $C\ne D$.
Hence neither $C$ nor  $D$ is a  tile. Both are not edges, because otherwise $Z \in C \cap D$, where the latter is an intersection of tiles and it has dimension $1$, which contradicts the definition of $\Omega$.  Thus either $C$ or $D$ is a side and contains the other. Therefore, $g(\F)\cap h(\F)$ is a side and hence, by part \ref{kappaIndependence} of Lemma~\ref{kappaElementaryL}, $\kappa_C(g,h)=[g\inv h]=\kappa_D(g,h)$, a contradiction.
This proves that indeed $\kappa_{Z}(g,h)$ is well defined.
By parts \ref{kappaAntisymmetric} and \ref{kappaTransitiveCell} of Lemma~\ref{kappaElementaryL} we have $\kappa_Z(g,h)=\kappa_Z(h,g)\inv$ and if $Z\in \cap_{i=1}^n g_i(\F)$ with $g_1,\dots,g_n\in \Gamma$ then
	\begin{equation*}
	 \kappa_Z(g_1,g_n)=\kappa_Z(g_1,g_2)\cdots \kappa_Z(g_{n-1},g_n).
	\end{equation*}

\begin{definition}
Let $\alpha$ be a continuous map $[0,1] \rightarrow \Omega$ as in Lemma~\ref{alfadaptexists} and let $\mathcal{L}=\left(a_0,g_1,a_1,g_2,\right.$ $\left.a_2 \ldots , g_n, a_n\right)$ be the partition of $\alpha$. We define
	$$\Phi (\mathcal{L}) = \kappa_{\alpha(a_1)}(g_1, g_2) \; \kappa_{\alpha(a_2)} (g_2,g_3) \cdots \kappa_{\alpha(a_{n-1})}( g_{n-1},g_n),$$
if $n\neq 1$. If $n = 1$, we set $\Phi(\mathcal{L})=1$.
\end{definition}

Observe that if $i\in \{1,\dots,n\}$ then $\alpha(a_i)\in g_i(\F)\cap g_{i+1}(\F)$. By the definition of $\Omega$, $ g_i(\F)\cap g_{i+1}(\F)$ has dimension $2$ or $3$. If it has dimension $3$, then $ g_i(\F)\cap g_{i+1}(\F)$ is a side containing $\alpha(a_i)$. If $ g_i(\F)\cap g_{i+1}(\F)$ has dimension $2$, then $\bigcap_{g,\; \alpha(a_i) \in g(\F)}g(\F)$ has dimension $2$ by the definition of $\Omega$ and thus, by Lemma~\ref{edgeunion}, $g_i(\F)\cap g_{i+1}(\F)$ contains an edge that contains $\alpha(a_i)$.  Hence $\alpha(a_i) \in C \subseteq  g_i(\F)\cap g_{i+1}(\F)$, for some cell $C$ and thus  $\kappa_{\alpha(a_i)}(g_i,g_{i+1})$ is well defined.

Let $Z, W \in \Omega$ and let $\mathbb{P}$ be the set of all the continuous maps $\alpha: [0,1] \rightarrow \Omega$ with $\alpha(0)=Z$ and $\alpha(1)=W$, and such that $\alpha$ verifies the conditions of Lemma~\ref{alfadaptexists}. The set $\P$ may be considered as a metric space with the metric determined by the infinite norm: $\Vert \alpha \Vert_{\infty} = \max \lbrace \vert \alpha(t) \vert \ : \ t \in [0,1] \rbrace$. We define the map
 $$\Phi: \P \rightarrow \Delta$$
by $\Phi(\alpha)=\Phi(\mathcal{L})$, where $\mathcal{L}$ is the partition of $\alpha$. This is well defined as by Lemma~\ref{alfadaptexists} the partition of $\alpha$ exists and is unique. The next lemma will be a crucial part in the proof of the injectivity of the map $\varphi$ defined in (\ref{delta}).

\begin{lemma}\label{PhiConstantL}
If both $Z$ and $W$ belong to the interior of some tile then  $\Phi:\P \rightarrow \Delta$ is constant.
\end{lemma}

\begin{proof}
We claim that it is sufficient to show that $\Phi$ is locally constant.
Indeed, assume this is the case and let  $\alpha,\beta\in \P$.
By Lemma~\ref{simplyconnected}, $\alpha$ and $\beta$ are homotopic in $\Omega$ and by Lemma~\ref{Eisele3}, there is a homotopy $H(t,-)$ in $\P$ from $\alpha$ to $\beta$.
Let $c$ denote  the supremum of the $s\in [0,1]$ for which $\Phi(H(s,-))=\Phi(\alpha)$.
Since, by assumption,  $\Phi$ is constant in a neighbourhood of $H(x,-)$, it easily follows that $c=1$ and thus $\Phi(\alpha)=\Phi(\beta)$.


To prove that $\Phi$ is locally constant, let $\alpha, \beta \in \P$ and let $\mathcal{L}_1=(a_0,g_1,a_1,\ldots ,g_n,a_n)$ and $\mathcal{L}_2$ be the partition of $\alpha$ and $\beta$ respectively. Moreover we denote by $d(-,-)$ the Euclidean distance.
Let $\lbrace 0<d_1 < d_2 <  \ldots < d_m <1 \rbrace$ be the sets of elements $d \in \left[0,1\right]$ such that  $\alpha(d) \in \partial g(\F)$ for some $g \in \h$ for every $1 \leq i \leq m$.
Lemma~\ref{alfcutfinite} ensures that this set is finite. Denote by $k_1, \ldots , k_{m+1}$ the elements in $\h$, such that $\alpha(d_{i-1},d_{i}) \subseteq k_i(\F)$ for $1 \leq i \leq m+1$, where we set $d_0=a_0=0$ and $d_{m+1}=a_n=1$. Observe that $k_1=g_1$, $k_{m+1}=g_n$ and $\lbrace k_1, \ldots, k_{m+1}\rbrace = \lbrace g_1, \ldots , g_n \rbrace$. Since $\F$ is locally finite, there is $\delta_1>0$ such that for every $i\in \{0,1,\dots,m+1\}$ and every $g\in \h$, if $B(\alpha(d_i),2\delta_1))\cap g(\F)\ne \emptyset$ then $\alpha(d_i)\in g(\F)$.
Since $\alpha$ is continuous there is $\epsilon<\min \left\{ \frac{d_i-d_{i-1}}{2}: i\in\{1,\dots,m+1\} \right\}$ such that, for every
$i\in \{ 0,1,\ldots ,  m+1\}$, $d(\alpha(t),\alpha(d_i))<\delta_1$ for every $t$ with $\vert t-d_i \vert<\epsilon$.
For every $i\in \{1,\dots,m\}$, let $d'_i=d_i-\epsilon$ and $d''_i=d_i+\epsilon$. We also set $d'_{m+1}=1$ and $d''_0=0$.
Observe that $d''_{i}\le d'_{i+1}$ for every $i\in \{0,\dots,m+1\}$.
Each $\alpha([d''_{i-1},d'_{i}])$ is compact and it is contained in $k_i(\F)^{\circ}$.
Again using that $\F$ is locally finite we obtain a positive number $\delta_2$ such that for every $i \in \lbrace 0, 1, \ldots, m+1 \rbrace$, $d(\alpha(t),g(\F))>\delta_2$ for every $t\in [d''_{i-1},d'_{i}]$ and every $g\in \Gamma$ with $g \neq k_i$.
Let $\delta=\min\{\delta_1,\delta_2\}$.

We will prove that if $\beta \in B_{\Vert \ \Vert_{\infty}}(\alpha,\delta)$, then $\Phi(\alpha)=\Phi(\beta)$.
So assume $\beta\in \P$ with $\Vert \alpha - \beta \Vert_{\infty} <\delta$.
Then $d(\alpha(t),\beta(t))<\delta$ for every $t\in [0,1]$.
In particular, as $d(\beta(t),\alpha(d_i))< 2\delta_1$,
\begin{equation}\label{Neighbour}
 \text{if } t\in (d'_i,d''_i) \text{ and } \beta(t)\in g(\F) \text{ then } \alpha(d_i)\in g(\F).
\end{equation}
Moreover, since $d(\alpha(t),\beta(t))<\delta_2$,
\begin{equation}\label{Closed}
 \text{if } t\in [d''_{i-1},d'_i] \text { and } \beta(t)\in g(\F) \text{ then } g=k_i.
\end{equation}
The interval $\left[0,1\right]$ may be written as
$$\left[0,d_1'\right] \cup \left[d_1',d_1''\right]  \cup \left[d_1'',d_2'\right] \cup \ldots \left[d_m',d_m''\right] \cup \left[d_{m}'',d_{m+1}\right].$$
Based on this information, we construct $\mathcal{L}_2$ and prove that $\Phi(\mathcal{L}_1)=\Phi(\mathcal{L}_2)$. By (\ref{Closed}), the elements $k_1, \ldots k_{m+1}$ appear in that order in $\mathcal{L}_2$. Between those elements may appear other elements $h \in \h$. For each $1 \leq i \leq m+1$, there are two possibilities: $k_i$ and $k_{i+1}$ are equal or they are different.
If $k_i=k_{i+1}$ and $k_i$ and $k_{i+1}$ are two consecutive elements in $\mathcal{L}_2$, then $\beta(\left(d_{i-1}'',d'_{i+1}\right)) \subseteq k_i(\F)=k_{i+1}(\F)$. According to the definition of the partition of $\beta$, $k_i$ and $k_{i+1}$ are represented by just one element in $\mathcal{L}_2$. Hence we may suppose, without loss of generality, that if $k_i$ and $k_{i+1}$ are two consecutive elements of $\mathcal{L}_2$, then they are different. Thus
$$\mathcal{L}_2 =(0,k_1,b_1, *_1 , k_2,b_2, *_2, \ldots, k_{m},b_m, *_{m}, k_{m+1},1),$$
where $d_i' < b_i < d_i''$ and $*_i$ represents a, possibly empty, sequence $(h_{i1},b_{i1},h_{i2},g_{i2}\ldots,h_{in_i},b_{in_i})$ with  $h_{ij} \in \h$ and $d_i' < b_{ij} < d_i''$ for every $1 \leq j \leq n_i$. As stated above, if $*_i$ is empty, then $k_i \neq k_{i+1}$. In that case $k_i=g_j$ and $k_{i+1}=g_{j+1}$ are two consecutive elements in $\mathcal{L}_1$ and $b_i=a_j$. Moreover
\begin{eqnarray*}
\beta(\left[d_{i-1}'',d_i' \right]) & \subseteq & k_i(\F), \\
\beta(\left[d_{i}'',d_{i+1}' \right])& \subseteq & k_{i+1}(\F), \\
\beta(b_i)& \in & k_i(\F) \cap  k_{i+1}(\F).
\end{eqnarray*}
By (\ref{Neighbour}) and the fact that $b_i \in \left[d_i',d_i''\right]$, for every $g \in \h$ such that $\beta(b_i)\in g(\F)$ we have that $\alpha(d_i)= \alpha(a_j) \in g(\F)$.
Hence, by Lemma~\ref{edgeunion}, $\bigcap_{g,\; \alpha(a_j) \in g(\F)} g(\F)$ is a cell, which is contained in every cell containing $\beta(b_i)$. Thus both $\alpha(a_j)$ and $\beta(b_i)$ are included in a same cell $C$ contained in $g_j(\F) \cap g_{j+1}(\F)$. Thus,
$$\kappa_{\alpha(a_j)}(g_j,g_{j+1}) =\kappa_C(g_j,g_{j+1})=\kappa_{C}(k_i,k_{i+1})=\kappa_{\beta(b_i)}(k_i,k_{i+1}).$$

Suppose now that $*_i$ is not empty. We will analyse the subsequence $$(k_i,b_i,h_{i1},b_{i1},h_{i2},b_{i2}\ldots,h_{in_i},b_{in_i},k_{i+1},b_{i+1}).$$
By lemma~\ref{edgeunion}, $\alpha(d_i) \in C$, $\beta(b_i) \in C_i$ and $\beta(b_{ij}) \in C_{ij}$ for $C$, $C_i$ and $C_{ij}$ cells for $1 \leq j \leq n_i$. As $b_i \in \left(d_i',d_i''\right)$ and also $b_{ij} \in \left(d_i',d_i''\right)$ for every $1 \leq j \leq n_i$, by (\ref{Neighbour}) we have that $\alpha(d_i) \in k_i(\F) \cap k_{i+1}(\F) \cap \bigcap_{j=1}^{n_i} h_{ij}(\F)$. Hence $C \subseteq C_i \cap \bigcap_{i=1}^{n_i} C_{ij}$. If $k_i \neq k_{i+1}$, then $k_i=g_l$, $k_{i+1}=g_{l+1}$ and $d_i=a_l$ for some $1 \leq l \leq n$ and by parts \ref{kappaIndependence} and \ref{kappaTransitiveCell} of Lemma~\ref{kappaElementaryL}
\begin{eqnarray*}
\kappa_{\alpha(a_l)}(g_l,g_{l+1})& = & \kappa_C(k_i,k_{i+1}) \\
 & = & \kappa_C(k_i,h_{i1})\kappa_C(h_{i1},h_{i2}) \ldots \kappa_C(h_{in_i}k_{i+1})\\
 & = & \kappa_{C_i}(k_i,h_{i1})\kappa_{C_{i1}}(h_{i1},h_{i2}) \ldots \kappa_{C_{in_i}}(h_{in_i}k_{i+1})\\
 & = & \kappa_{\beta(b_i)}(k_i,h_{i1})\kappa_{\beta(b_{i1})}(h_{i1},h_{i2}) \ldots \kappa_{\beta(b_{in_i})}(h_{in_i}k_{i+1})
 \end{eqnarray*}
If $k_i=k_{i+1}$, then $k_i=k_{i+1}=g_l$ for some $1 \leq l \leq n$. Then in $\Phi(\alpha)$, there is no $\kappa$-term corresponding to the subsequence above and by parts \ref{kappaIndependence} and \ref{kappaTransitiveCell} of Lemma~\ref{kappaElementaryL}, we have
\begin{eqnarray*}
\kappa_{\beta(b_i)}(k_i,h_{i1})\kappa_{\beta(b_{i1})}(h_{i1},h_{i2}) \ldots \kappa_{\beta(b_{in_i})}(h_{in_i}k_{i+1}) & = &
 \kappa_{C_i}(k_i,h_{i1})\kappa_{C_{i1}}(h_{i1},h_{i2}) \ldots \kappa_{C_{in_i}}(h_{in_i}k_{i+1})\\
  & = & \kappa_C(k_i,h_{i1})\kappa_C(h_{i1},h_{i2}) \ldots \kappa_C(h_{in_i}k_{i+1})\\
& = & \kappa_{C}(k_i,k_{i+1})\\
& = & \kappa_{C}(k_i,k_{i})\\
& = & 1.
\end{eqnarray*}
This shows that $\Phi(\mathcal{L}_1)=\Phi(\mathcal{L}_2)$.
\end{proof}

We are now ready to prove that $\varphi: \Delta \rightarrow \h$ (see (\ref{delta})) is injective. Proving the injectivity is equivalent to proving that if $g_1\ldots g_n=1$ with each $g_i \in \h$ and $\F \cap g_i(\F)$ a side, then $[g_1]\ldots[g_n]=1$ in $\Delta$. So, suppose that
$g_1\ldots g_n=1$ with each $g_i \in \h$ and $\F \cap g_i(\F)$ a side. For every $i=0,1,\ldots,n$, put
$h_i=g_1\ldots g_i$ and put $h_0=1$. Choose $Z \in \F^{\circ}$ and, for every $i=1,\ldots,n$, choose a path
from $h_{i-1}(Z)$ to $h_i(Z)$ which is only contained in $h_{i-1}(\F) \cup h_i(\F)$ and which is made up of a finite number of line segments, whose end-points do not belong to any element in $\M$. This is possible by Lemma~\ref{fpathconnected} and Lemma~\ref{Eisele1}. Let $\alpha$ be the path obtained by juxtaposing those paths. So
$\alpha(0)=Z=\alpha(1)$ and $\alpha$ is contained in the space $\P$. Let $\mathcal{L}_1=(a_0=0,h_0,a_1,\ldots, h_n,a_{n+1}=1)$ be the partition of $\alpha$. Let $\beta$ be the constant
path $\beta(t)=Z$ for every $t \in [0,1]$ and let $\mathcal{L}_2=(0,1,1)$ be the partition of $\beta$. Clearly $\beta$ is also contained in $\P$. Now we have that
\begin{eqnarray*}
[g_1] \ldots [g_n] & = & [h_0^{-1} h_1][h_1^{-1} h_2] \ldots [h_{n-1}^{-1} h_n]\\
 & = & \kappa_{a_1}(h_0,h_1) \kappa_{a_2}(h_1,h_2) \ldots \kappa_{a_n}(h_{n-1},h_n)\\
 & = & \Phi(\mathcal{L}_1)\\
 & = & \Phi(\alpha).
 \end{eqnarray*}
By Lemma~\ref{PhiConstantL}, $\Phi(\alpha)=\Phi(\beta)$. However $\Phi(\beta)=\Phi(\mathcal{L}_2)=\kappa_{1}(1,1)=1$ and
hence $[g_1]\ldots[g_n]=1$.

To sum up we have proven the following theorem.

\begin{theorem}\label{Poincarerelations}
Let $\F$ be the fundamental domain of $\h$ as defined above. Then the following is a presentation of $\h$:
\begin{itemize}
\item Generators: the pairing transformations of $\F$,
\item Relations: the pairing relations and the cycle relations.
\end{itemize}
\end{theorem}

It is well known that $\h$ is finitely presented. However it is not clear whether the presentation given in Theorem~\ref{Poincarerelations} is finite. The following proposition implies that at least the presentation is finite in case $R$ is a PID.

\begin{lemma}\label{sidesfinite}
If $R$ is a PID then every tile of $\F$ intersects only finitely many other tiles. In particular the fundamental domain $\F$ has  finitely many cells and thus finitely many sides and edges.
\end{lemma}

\begin{proof}
It is enough to prove that $\F$ intersects only finitely many other tiles. Moreover, if $1\ne \gamma\in \h$ then $\F\cap\gamma(\F)\subseteq \cup_{V\in \V_e} \F \cap V$ and $\V_e$ is finite by Theorem~\ref{FiniteSides}.
Thus it is enough to prove that if $V\in \V\cap \V_{\infty}$ then $X_V=\{\gamma\in \Gamma : V\cap \F \cap \gamma(\F) \ne \emptyset\}$ is finite.
	If $V\in \V$ then $V\cap F$ is compact, by Lemma~\ref{esscompact}. Since $\F$ is locally finite (Lemma~\ref{locallyfinite}), $X_V$ is finite.
So assume that $V\in \V_{\infty}$.
By (\ref{h0atmost1}), (\ref{fnotzero}) and Lemma~\ref{minh}, there are positive real numbers $a<b$ such that $B\times (b,+\infty) \subseteq \F \subseteq B\times (a,+\infty)$ (in the $(s_0,s_1,r,h)$-coordinates).
Let $K=B\times [a,b]$. If $\gamma\in X_V$ then either $K \cap \gamma(\F)\ne \emptyset$ or $V\cap (B\times (b,+\infty)) \cap \gamma(\F)\ne \emptyset$.
As $K$ is compact and $\te$ is locally finite, only finitely many $\gamma\in \Gamma$ satisfy the first condition.
On the other hand it is easy to see that $\R^2\times \R^+ \times (b,\infty) = \cup_{\gamma \in \Gamma_{\infty}} \gamma(B\times (b,+\infty))$ (in the $(s_1,s_2,r,h)$-coordinates).
Thus if $\gamma$ satisfies the second condition then $\gamma\in \Gamma_{\infty}$ and $\F\cap \gamma(\F)\ne \emptyset$.
It remains to prove that only finitely many elements of $\Gamma_{\infty}$ satisfy the last condition.
Fix $h_0>0$. Assume that $Z_1=\gamma(Z_2)\in \F_{\infty}\cap \gamma(\F_{\infty})$ with $\gamma\in\h_{\infty}$.
Let $Z_1'$ and $Z_2'$ be the elements of $\HTwo$ obtained by replacing the last coordinate of $Z_1$ and $Z_2$ by $h_0$.
Then $Z_1'=\gamma(Z_2') \in (B\times \{h_0\}) \cap \F_{\infty} \cap \gamma(\F_{\infty})$.
Thus $(B\times \{h_0\})\cap \gamma(\F_{\infty})\ne \emptyset$. As $B\times \{h_0\}$ is compact and $\F_{\infty}$ is a locally finite fundamental domain of $\Gamma_{\infty}$, only finitely many elements satisfies the last condition. This finishes the proof.
\end{proof}

\appendix

\section{Appendix: Topological Lemmas}

For completeness' sake we state some  results on real algebraic topology that have been  used in the previous sections.
The proofs of these  are probably well known to the specialists.
We would  like  to thank  Florian Eisele for making us aware of these and for providing us proofs.  Some of these results can be stated in greater generality for arbitrary differential manifolds, but their proofs need more sophisticated algebraic topology methods.



\begin{lemma}\label{pathconnected}
Let $\X$ be a path-connected open subset of $\R^n$.
Assume that if  $Z_0$ and $Z_1$ are distinct points in  $\X$, then there exists $\epsilon >0$ such that the open cylinder with axis $\left[Z_0,Z_1\right]$ and radius $\epsilon$ is contained $\X$. If $L$ is a locally finite collection of semi-algebraic varieties of dimension at most $n-2$ in $\X$ then $\X \setminus L$ is path-connected.
\end{lemma}

\begin{lemma}\label{simplyconnected}
Let $n\geq 3$ and let $\X$ be a connected open subset of $\R^n$. Assume $\pi_1(\X)=1$.
If  $\{T_i\}_i$ is a locally finite family of algebraic varieties in $\X$ of co-dimension at least $3$
then
	$$
		\pi_1\left(\X \setminus \bigcup_i T_i \right) = 1.
	$$
\end{lemma}

\begin{lemma}\label{connected}
Let $\overline{B}$ be a Euclidean closed ball in $\R^n$. If $\overline{B}=U_1 \cup U_2$, where $U_1$ and $U_2$ are closed semi-algebraic sets, with $U_1 \not \subseteq U_2$ and $U_2 \not \subseteq U_1$, then $dim(U_1 \cap U_2) \geq n-1$.
\end{lemma}

\begin{lemma}\label{Eisele1}
Let $\X$ be a connected open subset of $\R^n$.  If  $Z$ and $W$ are points in $\X$ then there exists a path $\alpha: \left[0,1\right] \rightarrow \X$ with $\alpha(0)=Z$ and $\alpha(1)=W$ and  such that it is built from  of a finite number of line segments that are parametrized by polynomials of degree at most 1.
Moreover, if   $\mathbb{Y}$ is  a subset of $\X$   with dense complement in $\X$ and such that $Z,W\not\in \mathbb{Y}$ then we can choose the end-points of the line segments in $\X \setminus \mathbb{Y}$.
\end{lemma}

\begin{lemma}\label{Eisele3}
Let $\X$ be a connected open subset of $\R^n$.
Let $\gamma$ and $\gamma'$ be two homotopic paths in $\X$ starting and ending at $x_0$ which are built from  line segments that  are parametrized by polynomials of degree at most 1.
Assume we are additionally given a subset $\mathbb{Y} \subset \X$ with dense complement in $\X$ and which satisfies the property that a line which is not wholly contained in $\mathbb{Y}$ has finite intersection with it.
If  $x_0 \not\in \mathbb{Y}$ then there is an integer $N$ and a homotopy
	$$
		H:\ [0,1]\times [0,1] \longrightarrow \X
	$$
from $\gamma$ to $\gamma'$ such that each loop $H(t,-)$ for $t\in [0,1]$ starts and ends in $x_0$ and
there exist $0 = \tau_1 < \tau_2 < \ldots < \tau_N=1$ such that $H(t,-)|_{[\tau_i, \tau_{i+1}]}$ is a line segment for each,
$t\in (0,1)$, which is  parametrized by polynomials of degree at most $1$ and at least one of the two points $H(t,\tau_i)$ and $H(t,\tau_{i+1})$ lies outside $\mathbb{Y}$.
\end{lemma}

\nocite{*}
\bibliographystyle{amsalpha}
\bibliography{biblioh2xh2-V3}

\vspace{.25cm}

\noindent Department of Mathematics, \newline Vrije
Universiteit Brussel,\newline Pleinlaan 2, 1050
Brussel, Belgium\newline
emails: efjesper@vub.ac.be and akiefer@vub.ac.be

\vspace{.25cm}

\noindent
Departamento de Matem\'{a}ticas,\newline
 Universidad de Murcia,\newline  30100 Murcia, Spain\newline
email: adelrio@um.es

\end{document}